\documentclass[12pt]{amsart}
\usepackage{amssymb}
\usepackage{amsmath}
\usepackage{mathabx}

\usepackage{geometry}
\usepackage{xcolor}
\definecolor{mygray}{gray}{0.8}

\usepackage{hyperref}

\usepackage{pstricks}
\usepackage{pst-plot}

\usepackage{enumitem}

\usepackage[latin1]{inputenc}
\def\toitself{\righttoleftarrow}

\def\Z{{\mathbb Z}}\def\T{{\mathbb T}}\def\R{{\mathbb R}}

\def\e{\eta}
\def\g{\gamma}
\def\s{\sigma}
\def\be{\beta}\def\d{\delta}
\def\be{\beta}
\def\th{\theta}

\def\D{\Delta}\def\G{\Gamma}

\def\b#1{\lbrace#1\rbrace}
\def\a#1{\left|#1\right|}

\def\l#1{\langle #1\rangle}

\def\<{\langle}
\def\>{\rangle}

\theoremstyle{plain}

\newtheorem{theo}{Theorem}[section]

\newtheorem{Main}{Theorem}

\newtheorem*{Main*}{Main Theorem}

\newcommand{\ph}{\varphi}

\def\a{\alpha}

\def\d{\delta}
\def\ti{\tilde}
\def\e{\varepsilon}
\def\pa{\partial}
\def\s{\sigma}
\def\th{\theta}

\let\newpf\proof \let\proof\relax

\newcommand{\ba}{\overline{A}}

\newcommand{\cF}{\mathcal{F}}

\def\be{\begin{equation}}
\def\ee{\end{equation}}

\def\ba{{\begin{align}}}
\def\ea{{\end{align}}}

\def\bm{\begin{pmatrix}}
\def\em{\end{pmatrix}}

\def\a{{\alpha}}

\def\g{{\gamma}}

\def\0{{\mathbf 0}}

\newtheorem{thm}{Theorem}[section]
\newtheorem{cor}[thm]{Corollary}

\newtheorem{lem}[thm]{Lemma}
\newtheorem{lemma}[thm]{Lemma}

\newtheorem{prop}[thm]{Proposition}

\theoremstyle{remark}
\newtheorem{rem}{Remark}[section]

\theoremstyle{definition}

\newenvironment{proof}{ \noindent{\it Proof.}\quad}{\ \hfill $\Box$\vskip .2cm}

\def\ssm{\smallsetminus}

\renewcommand{\setminus}{\ssm}

\newcommand{\N}{{\mathbb N}}

\newcommand{\h}{{\bf h}}

\def\B0{{\bold{0}}}
\def\b{\beta}
\def\a{\alpha}

\def\l{\lambda}
\def\be{\begin{equation}}
\def\ee{\end{equation}}
\def\cA{\mathcal{A}}
\def\cA{\mathcal{A}}

\def\cF{\mathcal{F}}
\def\cG{\mathcal{G}}
\def\cH{\mathcal{H}}

\def\cT{\mathcal{T}}

\catcode`\@=12

\def\Empty{}
\newcommand\oplabel[1]{
  \def\OpArg{#1} \ifx \OpArg\Empty {} \else
  	\label{#1}
  \fi}
		
\newcommand{\comm}[1]{}
\newcommand{\comment}[1]{}

\begin{document}

\title[On invariant circles accumulating separatrices]{On the accumulation of separatrices  by invariant circles}

\author{A. Katok${}^*$ and  R. Krikorian}

\address{
RK: Department of Mathematics, CNRS UMR 8088, 
CY Cergy Paris Université (University of Cergy-Pontoise),  2, av. Adolphe Chauvin F-95302 Cergy-Pontoise, France.} 
\email{raphael.krikorian@cyu.fr.}

\thanks{* A preliminary version of this paper was discussed by the authors some months before Anatole Katok passed away in April 2018.  \\ A. Katok was partially supported by NSF grant DMS 1602409. R. Krikorian was supported by a Chaire d'Excellence LABEX MME-DII, the project ANR BEKAM : ANR-15-CE40-0001 and an AAP project from CY Cergy Paris  Université.}


\maketitle
\begin{abstract}Let $f$ be   a smooth symplectic  diffeomorphism of $\R^2$  admitting a (non-split) separatrix  associated to  a hyperbolic fixed point. We prove  that if  $f$ is a perturbation of the time-1 map of a symplectic autonomous vector field, this separatrix is accumulated by a positive measure  set of  invariant circles. On the other hand, we provide examples of smooth symplectic diffeomorphisms  with a Lyapunov unstable non-split separatrix that are not accumulated by invariant circles.
\end{abstract}
\section{Introduction}

A theorem by M.R. Herman,``Herman's last geometric theorem'', cf. \cite{HermanICM98},  \cite{FayadKrikorian}),  asserts that if a    smooth orientation and  area preserving  diffeomorphism $f$  of the 2-plane $\R^2$ (or the 2-cylinder $\R/\Z\times\R$) admits a KAM circle $\Sigma$  (by definition, a smooth  invariant curve \footnote{Isotopic in $\R^2\setminus\{o\}$ to a circle centered at the origin in the case $f:\R^2\to\R^2$ or isotopic to $\R/\Z\times\{0\}$ in the cylinder case.} on which the dynamics of $f$ is conjugated to a Diophantine translation) then this KAM circle is accumulated by other KAM circles the union of which has  positive 2-dimensional Lebesgue  measure in any neighborhood of $\Sigma$. In this paper we investigate whether such a phenomenon holds if, instead of being a KAM circle, the invariant set $\Sigma$ is a {\it separatrix} of a hyperbolic fixed (or periodic) point of $f$. 

More precisely, 
we consider the following situation. Let $f:\R^2\to \R^2$, $f:(x,y)\mapsto f(x,y)$, $f(0,0)=(0,0)$ be   a smooth   diffeomorphism which is  symplectic w.r.t. the usual symplectic form $\omega=dx\wedge dy$ ($f^*\omega=\omega$).  We assume that $o:=(0,0)$ is a {\it hyperbolic} fixed point of $f$ (the matrix $Df(o)\in SL(2,\R)$ has distinct real eigenvalues)  and that there exists an $f$-invariant compact connected set $\Sigma\ni o$ such that $\Sigma\setminus\{o\}$ is a nonempty  connected  1-dimensional manifold  included in both  the stable and unstable manifolds $W^{s}_{f}(o)$, $W^{u}_{f}(o)$ associated to $o$:
$$\forall\ (x,y)\in \Sigma,\quad \ \lim_{n\to\pm\infty} f^n(x,y)=o.$$
Note that since $o$ is $f$-hyperbolic,  $\Sigma$ is homeomorphic to a circle and  $\Sigma\setminus\{o\}$ coincides with one of the two connected components of $W^s_{f}(o)\setminus\{o\}$ (resp. $W^u_{f}(o)\setminus\{o\}$).
We shall say that $\Sigma$ is a {\it separatrix} of $f$ associated to the hyperbolic fixed point $o$ or, without referring to the hyperbolic fixed point  $o$, that $\Sigma$ is a separatrix of $f$.

\begin{figure}
\begin{center}
\begin{pspicture}(-1,-1)(10,5)

\psbezier(0,0)(0,4)(0,4)(5,4)
\psbezier(0,0)(5,0)(6,0)(6,2)
\psbezier(6,2)(6,3)(6,4)(5,4)

\psline[linewidth=0.5pt, arrowsize=2pt 10, arrowlength=1,arrowinset=0.7]{>-}(0.5,0)(0,0)
\psline(-0.5,0)(0,0)
\psline[linewidth=0.5pt, arrowsize=2pt 10, arrowlength=1,arrowinset=0.7]{->}(0,0)(0,0.5)

\psline(0,-0.5)(0,0)

\rput(0.5,-0.5){$o$}
\rput(6.5,3.5){$\Sigma$}
\end{pspicture}
\end{center}
\caption{A (non-split) separatrix}
\end{figure}
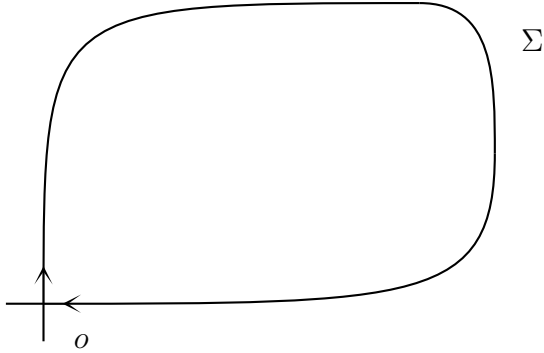

Examples of such diffeomorphisms $f$ can be obtained in  the following way. Let $X_{0}$ be a smooth autonomous {\it Hamiltonian} vector field of the form
\be X_{0}=J\nabla H_{0},\qquad J=\bm 0 & -1\\ 1& 0\em\label{eq:X0}\ee
where $H_{0}:\R^2\to\R$, of the form
$$
H_{0}(x,y)=Q_{0}(x,y)+O^3(x,y),\qquad  \l\in\R^*,
$$ 
(we can assume without loss of generality $\l>0$) is a smooth function. 
The time-1 map $f_{0}:=\phi_{X_{0}}^1$ of $X_{0}$
is a Hamiltonian (in particular symplectic) diffeomorphism of $\R^2$ admitting $o$ as a hyperbolic fixed  point. We assume that it has  a separatrix $\Sigma\ni o$ of the form
$$\Sigma\setminus\{o\}=\{\phi^t_{X_{0}}(p),\ t\in\R\},\qquad \textrm{for \ some}\ p\in \R^2\setminus\{o\}\ \textrm{s.t.}\quad \lim_{t\to\pm\infty}\phi^t_{X_{0}}(p)=o.$$

We now consider a  smooth time-dependent Hamiltonian vector field  $Y:\R/\Z\times \R^2\to \R$, $(t,(x,y))\mapsto Y(t,x,y)$ which is  1-periodic in $t$, symplectic w.r.t. $(x,y)$ and tangent to $\Sigma\setminus\{o\}$:
$$\forall\ t\in\R/\Z,\ \forall \ (x,y)\in \Sigma,\quad  \det(X_{0}(x,y),Y(t,x,y))=0.$$
One can for example choose $Y(t,x,y)=J\nabla F(t,x,y)$ where $F:\R/\Z\times \R^2\to\R$ is a smooth time-dependent Hamiltonian that satisfies
$$\forall\ t\in\R/\Z,\ \forall \ (x,y)\in \Sigma,\quad F(t,x,y)=F(t,0,0).$$
Note that since $o$ is a hyperbolic fixed point of $X_{0}$ one has for all $t$, $Y(t,o)=0$.
For $\e\in\R$ define the 1-periodic in $t$ symplectic vector field $\R^2\to\R^2$
\be X_{\e}^t(x,y):=X_{\e}(t,x,y)=X_{0}(x,y)+\e Y(t,x,y).\label{eq:1.1}\ee
For $\e$ small enough, the    time-0-to-1 map\footnote{If $X(t,z)$ is a time dependent vector field the time-$s$-to-$t$ map of $X$ is defined by $\phi^{t,s}_{X}(z(s))=z(t)$ for any $z(\cdot)$  solution  of $\dot z(t)=X(z(t))$. When $X$ is time-independent the notation $\phi^t_{X}$ stands for $\phi^{t,0}_{X}$. \label{footnote2}} 
\be f_{\e}=\phi^{1,0}_{X_{\e}}\label{deffepsilon}\ee of the symplectic vector field $X_{\e}$
 is a symplectic diffeomorphism of $\R^2$ admitting $o$ as a hyperbolic fixed point and still  $\Sigma$ as a separatrix. Note that $f_{\e}$ is a {\it Hamiltonian} diffeomorphism (for more details on Hamiltonian diffeomorphisms see \cite{Po}).

Here is the analogue of the aforementioned   Last Geometric Theorem of Herman:
\begin{Main}\label{MainA} \label{theo:A}
For any $r\in\N^*$ there exists $\e_{r}>0$, such that, for  any $\e\in ]-\e_{r},\e_{r}[$,     there exists a set of  $f_{\e}$-invariant $C^r$  KAM-circles 
accumulating the separatix $\Sigma$ and  which covers a set of positive Lebesgue measure of $\R^2$  in any neighborhood of $\Sigma$.
\end{Main}
Let us  clarify some points in the preceding statement.

By a {\it $C^r$ circle}, $r\geq 0$, we mean a $C^r$ non-self-intersecting closed curve (or equivalently, if $r\geq 1$, a nonempty compact connected 1-dimensional $C^r$ submanifold of $\R^2$) which is  isotopic in $\R^2\setminus\{o\}$ to the separatrix $\Sigma$. Such a set $\Gamma$ is invariant by $f_{\e}$ if $f_{\e}(\Gamma)=\Gamma$.

We say that a set $\cG$ of $f_{\e}$-invariant circles {\it accumulates} the set $\Sigma$ if for any $\xi>0$, the set of $\Gamma\in\cG$ such that ${\rm dist}(\Gamma,\Sigma)<\xi$ is not empty, where ${\rm dist}$ denotes the Hausdorff distance
$${\rm dist}(A,B)=\max\biggl (\sup_{a\in A}d(a,B),\sup_{b\in B}d(b,A)\biggr)
$$
(here $d(x,C)=\inf_{c\in C}\|x-c\|_{\R^2}$).

The $f_{\e}$-invariant circles obtained in Theorem \ref{theo:A} are  {\it  KAM circles}: the restrictions of $f_{\e}$ on each of these curves are $C^r$ circle diffeomorphisms that  are conjugated to  Diophantine\footnote{A real number $\a$ is Diophantine if there exist positive constants $\kappa,\tau$ such that for any $(p,q)\in\Z\times \N^*$, $|\a-(p/q)|\geq \kappa/q^{\tau}$. The constants  $\tau$ and $\kappa$ are respectively   called the {\it exponent}  and the {\it constant} of the Diophantine condition.  The set of Diophantine numbers with fixed exponent $\tau>2$ has full Lebesgue measure if the constant is not specified and positive measure if the constant is also fixed (and small).}  translations.  In our case, the exponent of the Diophantine condition can be choosen to be independent of $\e$ (it depends only on $\l$).  
\begin{rem}\label{rem:1.2}
 On the other hand, and this is a difference with the situation of Herman's Last Geometric Theorem, the {\it constants} of these Diophantine numbers  are  arbitrarily small. Moreover, as these circles accumulate the separatrix, their $C^2$-norm must explode.  

\end{rem}
\begin{rem}The phase space $\R^2$ can be replaced by the cylinder $\R/\Z\times \R$ in the statement of the Main Theorem.
\end{rem}
The smallness condition in Theorem \ref{theo:A}  is indeed necessary as shows the following theorem.

Let $\D_{\Sigma}$ be the bounded connected component of $\R^2\setminus\Sigma$.
\begin{Main}\label{theo:B}
There  exists a smooth  symplectic diffeomorphism $f:\R^2\to\R^2$  admitting 
a separatrix $\Sigma$ which is included  in  an open set $W$ of $\Sigma\cup \Delta_{\Sigma}$ that contains no $f$-invariant circle in $W\setminus\Sigma$.
\end{Main}
The situation described in the above Theorems \ref{theo:A}, \ref{theo:B} is  {\it not generic}. Indeed, as  Poincaré discovered, in general, the stable and unstable manifolds of a hyperbolic fixed or periodic  point of a symplectic map intersect transversally (one usually refers to this phenomenon as the {\it spiltting of separatrices}), a fact that forces the dynamics of $f$ to be ``quite intricate''. This was Poincaré's  key argument in his proof of the fact  that the Three-body problem in Celestial Mechanics 
 does not admit  a complete set of independent  commuting first integrals.   Later, Smale \cite{Smale} showed that this splitting of separatrices  has an even more  striking consequence on the dynamics of $f$, namely the existence of a {\it horseshoe}, that is, a uniformly hyperbolic $f$-invariant compact  set  (locally maximal) with positive topological entropy \footnote{By a result of the first author \cite{Katok}, in this situation,  positive topological entropy is indeed equivalent to the existence of a horseshoe. } and on which the dynamics of $f$ is ``chaotic'' (isomorphic to a 2-sided shift). 
  A consequence of the splitting of a separatrix is thus the existence of  a  {\it Birkhoff instability zone} (open region  without invariant circles) in the vicinity of this  split separatrix (see \cite{Herman-Ast1} for a detailed exposition on the topic).  In some sense, Theorem \ref{theo:A} shows  that in the perturbative situation (\ref{eq:1.1})-(\ref{deffepsilon}) ($\e$ small enough) the splitting of separatrices is essentially the only mechanism responsible  for the creation of instability zones. On the other hand, in a  ``non perturbative'' situation, Theorem \ref{theo:B} points in the opposite direction.  Figures \ref{fig:4}, \ref{fig:5} illustrate the role that plays  the smallness assumption in Theorem \ref{theo:A} (or its absence in Theorem \ref{theo:B}).

\bigskip\noindent{\it On the proofs of Theorems \ref{theo:A}, \ref{theo:B}.-- } As  suggests Remark \ref{rem:1.2}, the invariant circles of  Theorem \ref{theo:A} cannot be obtained directly {\it via} a classical KAM approach. On the other hand, the existence of the (non split) separatrix $\Sigma$ allows to associate to each diffeomorphism $f_{\e}$ a regular diffeomorphism $\mathring{f}_{\e}$, defined on a standard open annulus and preserving a finite probability measure, to which one can apply Moser's or R\"ussmann's  Invariant (or translated) curve Theorem \cite{Moser62}, \cite{Russmann} (see Section \ref{sec:6}). The thus obtained invariant curves for $\mathring{f}_{\e}$ yield invariant curves for $f_{\e}$. The construction of the diffeomorphism $\mathring{f}_{\e}$ is done as follows. We first make preliminary reductions involving some  Birkhoff  and symplectic Sternberg-like Normal Forms  (Section \ref{sec:2}) to have a control on the dynamics in some neighborhood of the hyperbolic fixed point $o$ (Section \ref{sec:3}). This allows us to define in Section \ref{sec:4} a first return map $\hat f_{\e}$ for $f_{\e}$, in a fundamental domain $\cF_{\e}$ the boundaries of which can be glued together to obtain an open {\it abstract} cylinder (or annulus).
This abstract cylinder can be {\it uniformized} to become a standard annulus and the first return map $\hat f_{\e}$  then becomes a regular diffeomorphism $\bar f_{\e}$ of a standard annulus (preserving some probability measure). This is done in Section \ref{sec:5}. We call {\it normalization} (see Section \ref{sec:5.3}) the uniformization operation and  we say that $\bar f_{\e}$ is the {\it renormalization}\footnote{The term {\it renormalization} in this paper has the same acceptation as in the theories of circle diffeomorphisms, holomorphic germs or quasi-periodic cocycles;  {\it cf.} \cite{Y}, \cite{Y-germs}, \cite{K}, \cite{AK1}.  } of $f_{\e}$. The dynamics of $\bar f_{\e}$ is closely related to that of $f_{\e}$ in the sense that the existence of invariant curves for $\bar f_{\e}$ translates into a similar statement for $f_{\e}$ (see Section \ref{sec:7}). The renormalized diffeomorphism $\bar f_{\e}$ has a {\it large twist} (this is a reminiscence of the hyperbolicity of  $f_{\e}$ at $o$) and we are thus led to {\it rescale} it to obtain the aforementioned diffeomorphism $\mathring{f}_{\e}$ which is now a small $C^r$-perturbation of an integrable twist map (this is where  the smallness assumption of Theorem \ref{theo:A} appears) with  {\it controled} twist (see Section \ref{sec:6}).  The proof of Theorem \ref{theo:A} is completed in Section \ref{sec:8}.

\medskip To prove Theorem \ref{theo:B} ({\it cf.} Section \ref{sec:9}) we construct a symplectic diffeomorphism $f$ (named $f_{pert}$ in that section) so that the associated  renormalized diffeomorphism $\bar{f}$ has an  orbit accumulating the boundary of the aforementioned annulus: this prevents the existence of $\bar{f}$-invariant  curves close to this boundary  and therefore of $f$-invariant curves close to the separatrix $\Sigma$.

\medskip   The technique\footnote{We note that the authors of \cite{TZ} introduce the ``separatrix map'' constructed by a glueing construction to investigate the size of the instability zones. Our approach here, which is  focused on a renormalization point of view, is different.} we use to prove   Theorem \ref{theo:A} might be useful to study  the dynamics  of symplectic twist maps with zero topological entropy (to which extent are they  integrable?\footnote{Angenent, \cite{Sib}, proves they are $C^0$-integrable in the sense that for any rotation number one can find a  $C^0$-invariant curve with this rotation number. Can one prove $C^k$-integrability? The word ``integrable'' is meant in a broad sense}) and the construction of  Theorem \ref{theo:B} might  give a hint to provide examples of {\it smooth} twist maps admitting {\it isolated} invariant circles with {\it irrational}\footnote{A modification of the example of Theorem \ref{theo:B} yields examples of such isolated invariant curves with {\it rational} rotation numbers. For the existence of curves with {\it irrational} rotation number  in {\it low regularity} and related results see \cite{Arnaud09}, \cite{Arnaud11}, \cite{Arnaud13}, \cite{AF}. } rotation number (if they exist, these curves bound two instability zones).

\subsection*{Acknowledgments} The second author (RK) wishes to thank  Bassam Fayad  and the referee for  their thorough reading of a preliminary version of this paper and their very useful comments.

\section{Normal Forms}\label{sec:2}
The main result of this Section is the following  Sternberg-like symplectic Normal Form theorem  (Proposition \ref{prop:2.5}) that will allow us in Section \ref{sec:3} to  control the {\it long-time} dynamics of $f_{\e}$ in a neighborhood of the hyperbolic point $o$. This will be useful when we shall define first return maps for $f_{\e}$ in convenient fundamental domain, see Section \ref{sec:4}.

\bigskip Let $f_{\e}$ defined by (\ref{deffepsilon}), (\ref{eq:1.1}).
\begin{prop}\label{prop:2.5}For any $k\in\N^*$ large enough, there exists $\e_{k}>0$ for which the following holds. There exist  a smooth family $(q_{\e,k})_{\e\in I}$ ($I\ni 0$ some open interval of $\R$) of polynomials $q_{\e,k}(s)=\l s+O(s^2)\in \R[s]$ and  a continuous family $(\Theta_{\e,k})_{\e\in I}$ of symplectic $C^k$-diffeomorphism of $\R^2$  such that $\Theta_{\e,k}(o)=o$, $D\Theta_{\e,k}(o)=id$ and  on a neighborhood $V_{k}$ of $o$ one has,  provided $\e\in ]-\e_{k},\e_{k}[$:
\begin{align}\textrm{on }\ V_{k},\qquad f_{\e,k}&\mathop{=}_{defin.}\Theta_{\e,k}\circ f_{\e}\circ \Theta_{\e,k}^{-1}\label{eq:2.3ante-1}\\
&=\phi^1_{J\nabla Q_{\e,k}},\quad \textrm{where }\quad Q_{\e,k}(x,y)=q_{\e,k}(xy) \label{eq:2.3ante}\end{align}
and
\be\textrm{on }\ V_{k},\qquad (\Theta_{0,k})_{*}X_{0}=J\nabla Q_{0,k}.\label{eq:2.3ante-2}\ee
\end{prop}

Note that $o$ is still a hyperbolic fixed point of $f_{\e,k}$ and that $$\Sigma_{\e,k}:=\Theta_{\e,k}(\Sigma)$$ is still a separatrix for $f_{\e,k}$.

\bigskip\noindent{\it Reduction of Theorem \ref{theo:A} to Theorem \ref{theo:2.1}.--}
After applying Proposition \ref{prop:2.5}  we are thus left with a family $(f_{\e,k})$ of $C^k$- symplectic diffeomorphisms, each $f_{\e,k}$ being conjugated to $f_{\e}$ and admitting a separatrix $\Sigma_{\e,k}$. Since the conclusions of Theorem \ref{theo:A} are clearly invariant by conjugation, to prove Theorem \ref{theo:A} we just need to prove that if $k\geq r$ and $\e$ is small enough, each separatrix $\Sigma_{\e,k}$ is accumulated by a set of positive measure of KAM-circles for $f_{\e,k}$. This is the content of Theorem \ref{theo:2.1} below, that we shall apply to the family  of $C^k$-diffeomorphisms $f_{\e,k}$ defined by (\ref{eq:2.3ante-1}), (\ref{deffepsilon}), (\ref{eq:1.1}), but that  holds for any  family  (that we still denote in what follows  $(f_{\e})_{\e\in I}$ to alleviate the notations) of symplectic $C^k$-diffeomorphisms satistying the following hypothesis.

\medskip 
Let $(f_{\e})_{\e\in I}$,   ($I\ni 0$ open interval of $\R$) be a family of $C^k$-symplectic diffeomorphisms of $\R^2$ that satisfies:
\begin{enumerate}[label=($H$\arabic*)]
\item\label{iii0} Each $f_{\e}$ has a (non split) separatrix $\Sigma_{\e}$ associated to the hyperbolic point $o$.
\item\label{iii1} The map $I\ni\e\to f_{\e}-id\in C^k(\R^2,\R^2)$ is continuous (the norm on $C^k$ is the usual $C^k$-norm);
\item\label{iii2} On some neighborhood $V$ of $o$, each $f_{\e}$ coincides with the time-1 map of a symplectic vector field $J\nabla Q_{\e}(x,y)$ where $Q_{\e}(x,y)=q_{\e}(xy)$, $q_{\e}\in C^{k+1}(\R^2)$
$$q_{\e}(t)=\l t+O^2(t),\qquad \l>0.$$
\item\label{iii3} On $\R^2$,  $f_{0}=\phi^1_{X_{0}}$ where $X_{0}=J\nabla H_{0}$ is a hamiltonian vector field that coincides with $J\nabla Q_{0}$ on $V$. 
\end{enumerate}

\begin{rem}\label{rem:2.1}On $V_{}$ the orbits of $f_{\e}{}_{|\ V}= \phi^1_{J\nabla Q_{\e}}$  are pieces of hyperbolae $\{xy={\rm cst}\}$ (condition \ref{iii2}).

When $\e=0$, for any $z\in\{xy=c\}\cap V$, $N\in\Z$ such that $f_{0}^N(z)\in V$ one has $f_{0}^N(z)\in \{xy=c\}\cap V$ (condition \ref{iii3}).  
\end{rem}

\begin{rem}\label{rem:2.2}The intersection $\Sigma_{\e}\cap V_{}$ is the union 
$$\Sigma_{\e}\cap V_{}=(W^s_{f_{\e}}(o)\cap  V_{})\cup (W^{u}_{f_{\e}}(o)\cap V_{})$$ and 
$$W^{s}_{f_{\e}}(o)\cap V_{}=(\R\times\{0\}) \cap V_{},\qquad W^{u}_{f_{\e}}(o)\cap  V_{}=(\{0\}\times \R) \cap V_{}.$$
\end{rem}

One then has:

\begin{theo}\label{theo:2.1}There exists $k_{0}\in\N$ for which the following holds. Let  $k\geq k_{0}+2$ and let  $(f_{\e})_{\e\in I}$ be  a family of $C^k$-symplectic diffeomorphisms of $\R^2$ satisfying the previous conditions \ref{iii0}, \ref{iii1},  \ref{iii2}, \ref{iii3}.  Then, there exists $\e_{1}>0$ such that for any $\e\in ]-\e_{1},\e_{1}[$ the diffeomorphism $f_{\e}$ 
 admits  a set of positive Lebesgue measure of invariant $C^{k-k_{0}-2}$-circles in any neighborhood of the separatrix $\Sigma_{\e}$. 
 
 Moreover, if $k-k_{0}-2\geq k_{1}$ ($k_{1}$ depending only on $\l$) these circles are KAM-circles.
\end{theo}
We shall give the proof of Theorem \ref{theo:2.1} in Section \ref{sec:8}.

\bigskip The proof of Proposition \ref{prop:2.5} occupies the rest of this Section. It will be based on a first reduction obtained by  performing some steps of Birkhoff Normal Forms (Proposition \ref{BNFbis}) and then the application of various Sternberg like Normal Forms (Corollary \ref{cor:2.3} and Proposition \ref{prop:2.7}).

\subsection{Birkhoff Normal Form for the time-periodic vector field $X_{\e}^t$}
A preliminary step in  Sternberg's classical  Linearization Theorem is to first conjugate the considered   system (diffeomorphism or vector field) defined in the neighborhood of the hyperbolic fixed point $o$ to a system which is tangent to an integrable model to some high enough order. This is what we do  in this subsection and in a symplectic framework (see Proposition \ref{BNFbis}) by using Birkhoff Normal Form techniques.  
\subsubsection{Periodically forced  vector fields}\label{sec:2.1.1}
 Let $X:\R\times\R^2\ni (t,x)\mapsto X^{t}(z):=X(t,z)\in\R^2$  be a smooth  time dependent  symplectic vector field:  for each $t$  the 1-form $i_{X_{t}}\omega$ is closed (and hence locally exact). For $t,s\in\R$ we denote by $\phi^{t,s}_{X}$ the flow of $X$ between times $s$ and $t$ when it is defined (see  footnote \ref{footnote2} for the definition of $\phi^{t,s}_{X}$). 
If $t\mapsto g^{t}(\cdot)$ is a one-parameter family of symplectic diffeomorphisms one has 
\be  g^{t}\circ \phi_{X}^{t,s}\circ (g^{s})^{-1}=\phi^{t,s}_{\ti X}\label{*1}\ee
where $\ti X:(t,z)\mapsto \ti X^{t}(z):=\ti X(t,z)$  is the  smooth  time dependent  symplectic  
vector field
\be\ti X^{t}=\pa_{t}g^{t}\circ (g^{t})^{-1}+(g^{t})_{*}X^{t}.\label{*2}
\ee
Conversely, if (\ref{*2}) is satisfied then so is  (\ref{*1}). Note that if $g^t$ depends 1-periodically on $t$ then  (\ref{*1}) yields    the   more classical  conjugation equation 
$$g\circ \phi^{1,0}_{X}\circ g^{-1}=\phi_{\ti X}^{1,0}
$$
where $g=g^0=g^1$ ($g^t$ is 1-periodic in $t$).

Assume now  that  $X^{t}$ depends  1-periodically in $t$ and in a smooth way  on a small parameter $\e\in\R$; we furthermore assume that it is of the form  
\be X_{\e}^t(z)=J\nabla H_{\e}^t(z)\label{2.4*}\ee where ($z=(z_{1},z_{2})\in\R^2$)
 \be H_{\e}^{t}(z)=\l_{\e}(t)z_{1}z_{2}+O^3(z),\qquad \int_{\T}\l_{\e}(t)dt>0,\qquad \l_{0}(t)=\l\in\R^*_{+},\label{defH}\ee
$H_{\e}:\R/\Z\times \R^2\to \R$, $H_{\e}:(t,z)\mapsto H_{\e}(t,z):=H_{\e}^{t}(z)$ being  a smooth function.
Assume also that for some $j\in\N^*$
$$g^t_{\e}(z)=\phi^{1}_{J\nabla G_{\e}^t}(z)=id+O^j(z),\qquad G_{\e}^t(z)=O^{j+1}(z)$$
where   $G:I\times \R/\Z\times(\R^2,o)\ni (\e,t,z)\mapsto G_{\e}(t,z):=G_{\e}^{t}(z)\in\R$ is a smooth function. 
Then, one has 
$$\pa_{t}g_{\e}^{t}\circ (g_{\e}^{t})^{-1}=J\nabla \pa_{t}G_{\e}^t+O^{j+1}(z)
$$
$$(g_{\e}^{t})_{*}X_{\e}^{t}=J\nabla H_{\e}^t\circ (g_{\e}^{t})^{-1}=J\nabla H_{\e}^{t}+J\nabla\{G_{\e}^{t},H_{\e}^t\}+O^{j+1}(z)
$$
(here  $\{A,B\}$ denotes the Poisson bracket $\{A,B\}=\<\nabla A,J\nabla B\>$) so that $\ti X_{\e}^t$ defined by (\ref{*2}) is of the form 
\be \ti X_{\e}^t= J\nabla \ti H_{\e}^t\label{2.5*}\ee with 
\begin{align}\ti H_{\e}^t&=H_{\e}^t+\pa_{t}G_{\e}^{t}+\{G_{\e}^{t},H_{\e}^{t}\}+O^{j+2}(z)\label{2.6*}\\
&=H_{\e}^t+\pa_{t}G_{\e}^t+\{G_{\e}^t,H_{2,\e}^t\}+O^{j+2}(z)\label{3*}
\end{align}
where we have denoted $H_{2,\e}^t(z_{1},z_{2})=\l_{\e}(t)z_{1}z_{2}$.

If in the preceding equation one chooses $G^{t}_{\e}=G_{\e,2}^t$ with $G_{\e,2}^t(z)=a_{\e,0}(t)z_{1}z_{2}$ where $a_{\e,0}$ is the 1-periodic function defined by 
$$a_{\e,0}(t)=-\int_{0}^t\biggl(\l_{\e}(s)-\int_{\T}\l_{\e}(u)du\biggr)ds
$$
one has
$$ \ti H_{\e}^t(z)=\bar\l_{\e}z_{1}z_{2}+O^3(z)$$
where $\bar \l_{\e}=\int_{\T}\l_{\e}(t)dt$. In other words, performing a change of coordinates  (\ref{*2}) on $X_{\e}^t$ with $g^t_{\e}=g_{\e,2}^t=\phi^{1}_{J\nabla G^t_{\e,2}}$ we can assume that in (\ref{defH}) $\l_{\e}(t)$ does not depend on $t$
\be H_{\e}^t(z)=\l_{\e}z_{1}z_{2}+O^3(z),\qquad \l_{\e}\in\R^*_{+}\label{2.9*}\ee
(we write $\l_{\e}$ in place of $\bar \l_{\e}$).

\subsubsection{Birkhoff Normal Form}
Having put $H_{\e}^t$ under the form (\ref{2.9*}), we now eliminate  by successive conjugations (\ref{*2})  {\it non-diagonal} higher order terms in $z$ from $H_{\e}^t$ (note that they depend on $t$).

\medskip The following lemma describes this elimination procedure.
\begin{lemma}\label{lemma:Birkhoff}Let $j\in\N$, $j\geq 2$. Assume that for some polynomials $q_{\e}(s)=\l s +O(s^2)\in \R[s]$ of degree $\leq [ j/2]$ depending smoothly on $\e$
$$H_{\e}^t(z)=q_{\e}(z_{1}z_{2})+O^{j+1}(z).$$
Then, there exist a smooth  family $(\ti q_{\e})_{\e}$ of polynomials  $\ti q_{\e}(s)=\l s+O(s^2)\in \R[s]$ of degree $\leq [(j+1)/2]$ and a smooth  family of smooth maps $G_{\e}:\R/\Z\times (\R^2,o)\ni (t,z)\to G_{\e}(t,z)=G^t_{\e}(z)\in\R^2$ such that   on a neighborhood of $o$
\be \begin{cases}&G_{\e}^t(z)=O^{j+1}(z)\\
&H_{\e}^t(z)+\pa_{t}G_{\e}^t(z)+\{G_{\e}^t,H_{\e}^t\}(z)=\ti q_{\e}(z_{1}z_{2})+O^{j+2}(z).\end{cases}\label{eq:2.8}\ee
Moreover, if for $\e=0$, $H_{0}^t$ does not depend on $t$, one can choose $G_{0}^t$ to be independent of $t$.
\end{lemma}
\begin{proof}See the Appendix, Section \ref{appendix:Birkhoff+Sternberg}.
\end{proof}

\bigskip 
Let now $X_{\e}^{t}$ be the family of vector fields of (\ref{eq:1.1}).
\begin{prop}\label{BNFbis}For any $N\geq 1$ there exist  an open neighborhood $V_{N}$ of $o$, a  smooth two-parameters family $(b_{\e}^t)_{\e\in I,t\in\R/\Z}$ ($I$ some open interval containing 0) of smooth symplectic diffeomorphisms $b_{\e}^t:(\R^2,o)\toitself$ satisfying   $b_{\e}^{t}(o)=o$, $Db_{\e}^t(o)=id$ and a smooth family of polynomials $q_{\e,N}(s)=\l s+O(s^2)$ of degree $\leq [(N+1)/2]$, such that for any $\e\in I$, $t\in\R/\Z$, $(x,y)\in V_{N}$ one has
\begin{align*}X_{\e}^{(1),t }&\mathop{=}_{defin.}\ (\pa_{t}b^t_{\e})\circ (b^t_{\e})^{-1}+(b^t_{\e})_{*}X_{\e}^t\\ &=J\nabla Q_{\e,N}+O^{N+1}(x,y)
\quad \textrm{with }\quad  Q_{\e,N}(x,y)=q_{\e,N}(xy)\end{align*}
and for $\e=0$, $b^t_{0}$ is independent of  $t$.

\end{prop}
\begin{proof}
Applying the preceding Lemma \ref{lemma:Birkhoff} and relations (\ref{*2})-(\ref{2.6*}) inductively (starting from (\ref{2.9*})) we thus construct polynomials  $q_{\e,j}$ of degree $\leq[ j/2]$  ($j\geq 2$) and functions $G_{\e,j}^t=O^{j+1}(z)$ such that if one defines
$$b^t_{\e}=g^t_{\e,N}\circ\ldots\circ g^t_{\e,2}=id+O^2(z),\qquad g^t_{\e,j}=\phi^1_{J\nabla G^t_{\e,j}}=id+O^{j}(z)$$
one has 
\begin{align*}\ti X_{\e}^{t}:&=\pa_{t}b_{\e}^{t}\circ (b_{\e}^{t})^{-1}+(b_{\e}^{t})_{*}X_{\e}^{t}\\
&=J\nabla Q_{\e,N}+O^{N+1}(z),\quad  \textrm{with}\quad Q_{\e,N}(z)=q_{\e,N}(z_{1}z_{2})
\end{align*}
all dependences on $\e$ being smooth. Moreover,  if $X_{0}^t$ is independent of $t$ the diffeomorphism $b_{0}^t$ is independent of $t$. 
\end{proof}
\begin{rem}Note that since $b_{0}^t\equiv b_{0}$ is independent of $t$,  the vector field 
$$X_{0}^{(1)}=(b_{0})_{*}X_{0}$$
is autonomous.
\end{rem}
\subsection{Symplectic Sternberg theorem for the autonomous vector field $X_{0}^{(1)}$ }
We shall need a {\it symplectic} version of the famous theorem by S. Sternberg  (on smooth linearization of  hyperbolic germs of smooth vector fields, see \cite{Sternberg}) as proved in \cite{BLW} or \cite{Chaperon} (see also \cite{Sternberg2}). We follow here the exposition of \cite{BLW}.

 Let $Z_{i}$, $i=1,2$  be two symplectic smooth autonomous vector fields such that for some $\l\in\R^*$ and $N\in\N$ one has
\be\begin{cases}& Z_{i}(x,y)=-\l x\frac{\pa}{\pa x}+\l y\frac{\pa}{\pa y}+O^2(x,y)\ (i=1,2)\\
&Z_{1}(x,y)-Z_{2}(x,y)=O^{N+1}(x,y).
\end{cases}.\label{BNFZ}
\ee
\begin{theo}[\cite{BLW}, Theorem 1.2]\label{prop:SternbergZbis}There exist positive constants $A,B$ for which the following holds. Let $m\in\N^*$ large enough and $N=[(m+B)/A]+1\geq 1$. If (\ref{BNFZ}) is satisfied then there exists a $C^{m}$ symplectic change of coordinates $S_{0}:(\R^2,0)\toitself$ such that on a neighborhood of $o$
\be \begin{cases}&(S_{0})_{*}Z_{1}=Z_{2}\\
&S_{0}(o)=0,\quad DS_{0}(o)=id.\end{cases}\label{2.7}
\ee
\end{theo}

We apply the preceding theorem to the case $Z_{1}=X_{0}^{(1)}$ and $Z_{2}=J\nabla Q_{0,N}$ ($X_{0}^{(1)}$, $Q_{0,N}$ given by Proposition \ref{BNFbis} when $\e=0$). In view of Proposition \ref{BNFbis}, the condition (\ref{BNFZ}) is satisfied and we hence get a symplectic diffeomorphism $S_{0}$ satisfying $S_{0}(o)=0$, $DS_{0}(o)=id$ and  such that on a neighborhood of $o$
$$(S_{0})_{*}X_{0}^{(1)}=J\nabla Q_{0,N}.$$
For each value of $t\in\R/\Z$ and $\e\in I$, the diffeomorphism $(S_{0}\circ b^t_{\e})$ fixes the origin and its derivative at the origin is the identity. It can thus be extended as a symplectic $C^{m}$-diffeomorphism $R^t_{\e}$ of $\R^2$ ({\it cf.} Lemma \ref{lemma:extension}). Note that the dependence of $R_{\e}^t$ w.r.t.  $t$ is smooth and 1-periodic ($t\in\R/\Z$).  We now define on $\R/\Z\times \R^2$ the time-periodic vector field $X_{\e}^{(2)}:(t,(x,y))\in (\R/\Z)\times\R^2\to\R^2$ by 
\be X_{\e}^{(2),t}\mathop{=}_{defin.}(\pa_{t}R^t_{\e})\circ (R^t_{\e})^{-1}+(R^t_{\e})_{*}X_{\e}^t\label{2.8}\ee
and we observe that on a neighborhood of $o$
$$(R^t_{0})_{*}X_{0}=X_{0}^{(2)}=J\nabla Q_{0,N}.$$

Since the conjugacy relation (\ref{2.8}) is equivalent to (see subsection \ref{sec:2.1.1})
$$\forall\ t,s,\quad R^t_{\e}\circ \phi^{t,s}_{X_{\e}}\circ (R_{\e}^s)^{-1}=\phi^{t,s}_{X^{(2)}_{\e}}$$
we get by taking $t=1$, $s=0$ and setting $R_{\e}:=R^1_{\e}=R^0_{\e}$, the following corollary:
\begin{cor}\label{cor:2.3}If $m\in\N^*$ is large enough and $N=[(m+B)/A]+1$, there exists a smooth family $(R_{\e})$ of $C^m$ symplectic  diffeomorphisms of $\R^2$ such that $R_{\e}(o)=o$, $DR_{\e}(o)=id$ and on a neighborhood of $o$
\begin{align}f_{\e}^{(1)}:&\mathop{=}_{defin.}R_{\e}\circ f_{\e}\circ (R_{\e})^{-1}\label{2.14*}\\
&=\phi^1_{J\nabla Q_{\e,N}}+O^{N+1}(x,y).\label{2.15*}\end{align}
Moreover
\be (R_{0})_{*}X_{0}=J\nabla Q_{0,N}.\label{2.15**}\ee
\end{cor}
Note that the last equation shows that
\be f_{0}^{(1)}=\phi^1_{J\nabla Q_{0,N}}.\label{f01}\ee

\subsection{Symplectic Sternberg Normal Form for the diffeomorphism $f_{\e}^{(1)}$}
Theorem \ref{prop:SternbergZbis}  has a version for  smooth germs of symplectic  diffeomorphisms which are hyperbolic at the origin, this is Theorem 1.1 of \cite{BLW}. In our paper we shall need a parametric version of that result, which is not explicitly stated in \cite{BLW} but that can be checked after   close examination of the proof. 
\begin{prop}\label{prop:2.7}There exist constants $A_{1},B_{1}$ depending on $\l\in\R^*$ such that the following holds. Let $m\in\N^*$ large enough  and $N=[m/2]-3$.     If $(g_{1,\e})_{\e\in I}$ and $(g_{2,\e})_{\e\in I}$ ($I\ni 0$ some open interval of $\R$)  are two continuous (w.r.t.  $\e\in I$) families of $C^m$ symplectic diffeomorphisms $(\R^2,o)\toitself$ such that 
\be \begin{cases}&\forall \e,\quad g_{1,\e}(o)=g_{2,\e}(o)=o\\
&Dg_{1,0}(o)={\rm diag}(\l,\l^{-1}) \quad \textrm{(is \ hyperbolic)}\\
&g_{1,\e}(x,y)=g_{2,\e}(x,y)+O^{N+1}(x,y)\\
&g_{1,0}=g_{2,0}
\end{cases}\label{assumpg12}\ee
then, there exists a continuous family $(S^{(1)}_{\e})_{\e}$ (w.r.t. $\e\in I$ small enough),  of $C^k$ symplectic diffeomorphisms such that $S^{(1)}_{\e}(o)=o$, $DS^{(1)}_{\e}(o)=id$ with $k=[NA_{1}-B_{1}]-1$ and 
$$\begin{cases}&S^{(1)}_{\e}\circ g_{1,\e}\circ (S^{(1)}_{\e})^{-1}=g_{2,\e}\\
&S^{(1)}_{0}=id.\end{cases}$$
\end{prop}

\subsection{Proof of Proposition \ref{prop:2.5} }\label{sec:2.4}
It will be a consequence of Corollary \ref{cor:2.3} and Proposition \ref{prop:2.7}.

\bigskip 
We  first choose $N$ so that $k=[NA_{1}-B_{1}]-1$ and we define $m$ by $N=[(m+B)/A]+1$. If $k$ is large enough $m$ will satisfy the assumption of Corollary \ref{cor:2.3}. We then apply Proposition \ref{prop:2.7} to $g_{1,\e}=f_{\e}^{(1)}$, $g_{2,\e}=\phi^1_{J\nabla Q_{\e,N}}$ which satisfy (\ref{assumpg12}) (note that (\ref{2.15*}) is satisfied). This provides us with a continuous  family $(S_{\e}^{(1)})_{\e}$ of $C^k$ symplectic diffeomorphisms defined in a fixed neighborhood of $o$  such that  $S^{(1)}_{\e}(o)=o$, $DS^{(1)}_{\e}(o)=id$ and on a neighborhood of $o$
\be \begin{cases}&S^{(1)}_{\e}\circ f_{\e}^{(1)}\circ (S^{(1)}_{\e})^{-1}=\phi^1_{J\nabla Q_{\e,N}}\\
&S^{(1)}_{0}=id.\end{cases}\label{2.21*}\ee
We can extend these $S^{(1)}_{\e}$ as symplectic $C^k$ diffeomorphisms $S^{(2)}_{\e}$ of $\R^2$ which depend continuously on $\e$ ({\it cf.} Lemma \ref{lemma:extension}). We then define 
$$\Theta_{\e,k}=S_{\e}^{(2)}\circ R_{\e}$$
and we observe that on a neighborhood of $o$
$$\begin{cases}
&\Theta_{\e,k}\circ f_{\e}\circ \Theta_{\e,k}^{-1}=\phi^1_{J\nabla Q_{\e,N}}\\
&(\Theta_{\0,k})_{*}X_{0}=J\nabla Q_{0,N};
\end{cases}$$
indeed, the first equality comes from  (\ref{2.14*} and the first equation of (\ref{2.21*}), while the second is a consequence of (\ref{2.15**}) and the second equation of (\ref{2.21*}).

To conclude the proof we rename $q_{\e,N}$, $Q_{\e,N}$ as $q_{\e,k}$, $Q_{\e,k}$.

\hfill$\Box$

\vskip 1cm\noindent{\bf Note:} From now on, and till the end of Section \ref{sec:8}, we shall work in the setting   of Theorem \ref{theo:2.1} with a family of $C^k$ symplectic diffeomorphisms satisfying conditions \ref{iii0}, \ref{iii1},  \ref{iii2}, \ref{iii3}.

\section{Dynamics in a neighborhood of the origin}\label{sec:3}
The purpose of this Section is to estimate  the time spent by the orbits of the  flow $\Phi^t_{J\nabla Q_{\e}}$ in the neighborhood $V_{}$  of the hyperbolic point $o$. 

\smallskip To do that, we perform one more   change of coordinates.

\medskip
 Let us define the following diffeomorphisms $\Xi_{1},\Xi_{2}$
\be\begin{cases}& \forall\ (x,y)\in\ \R^*_{+}\times\R,\quad  \Xi_{1}(x,y)=(\ln x,xy),\\
&\forall\ (x,y)\in\  \R\times\R^*_{+}, \quad \Xi_{2}(x,y)=(-\ln y,xy).
\end{cases}\label{defXi}
\ee
Since  $d(\ln x)\wedge d(xy)= d(-\ln y)\wedge d(xy)=dx\wedge dy$ we see that $\Xi_{i}$, $i=1,2$, are symplectic.

Let $I_{1},I_{2}\subset \R^*_{+}$ be
 some open intervals such that  
$I_{1}\times\{0\}$ and $\{0\}\times I_{2}$ are both contained in $V_{}$.

\begin{lem}\label{prop:linearization}
Let $(x_{*},y_{*})\in (I_{1}\times\R)\cap  V_{}$ 
and $\bar t_{I_{2}}(x_{*},y_{*})=\inf\{t>0:\phi_{J\nabla Q_{\e}}^t(x_{*},y_{*})\in (\R\times I_{2})\cap  V_{}\}$. Then: \begin{enumerate}\item  There exists $c(I_{1},I_{2})\geq 1$ such that  if $0<x_{*}y_{*}\lesssim_{I_{1},I_{2},\l} 1$ one has 
\be c(I_{1},I_{2})^{-1} \frac{|\ln (x_{*}y_{*})|}{\l}\leq \bar t_{I_{2}}(x_{*},y_{*})\leq c(I_{1},I_{2}) \frac{|\ln (x_{*}y_{*})|}{\l}. \label{esttstarbis}
\ee
\item   For any $(x,y)$ in a neighborhood of $(x_{*},y_{*})$ and any  $t$ in a neighborhood of $\bar t_{I_{2}}(x_{*},y_{*})$
\be \Xi_{2}\circ\phi_{J\nabla Q_{\e}}^{t}\circ \Xi_{1}^{-1}:(u,v)\mapsto (u+\tau_{\e}^t(v),v)\label{eq:3.10}\ee
with 
\be  \tau_{\e}^t(v)=tq_{\e}'(v)-\ln v.\label{eq:3.11bis}
\ee
\end{enumerate}

\end{lem}
\begin{proof} 
1) We evaluate $\bar t_{I_{2}}(x_{*},y_{*})$. Since 
$$\phi_{J\nabla Q_{\e}}^t(x_{*},y_{*})=(e^{-tq'_{\e}(x_{*}y_{*})}x_{*},e^{tq'_{\e}(x_{*}y_{*})}y_{*})$$
we   have 
$e^{tq_{\e}'(x_{*}y_{*})}y_{*}\in I_{2}$ if and only if  
$$t\in \biggl]\frac{\ln((x_{*}y_{*})^{-1} \times x_{*}\min I_{2})}{q'_{\e}(x_{*}y_{*})},\frac{\ln((x_{*}y_{*})^{-1}\times x_{*}\max I_{2})}{q'_{\e}(x_{*}y_{*})}\biggr[.
$$
 Hence for $x_{*}y_{*}$ small enough
$$\biggl|\bar t_{I_{2}}(x_{*},y_{*})-\frac{|\ln(x_{*}y_{*})|}{q'_{\e}(x_{*}y_{*})}\biggr|\leq \frac{\max(|\ln(x_{*}\min I_{2}  )|  ,|\ln(x_{*}\max I_{2}  )|) }{q'_{\e}(x_{*}y_{*})}.$$
Since for $0<x_{*}y_{*}\lesssim 1$ one has $q'_{\e}(x_{*}y_{*})\asymp \l$, there exists $c(I_{1},I_{2})$ such that  if $x_{*}y_{*}$ small enough (how small depending on $I_{1},I_{2},\l$) the inequality (\ref{esttstarbis}) is satisfied.

\medskip\noindent
2) We write
$$\Xi_{2}\circ\phi_{J\nabla Q_{\e}}^t\circ \Xi_{1}^{-1}=\Xi_{2}\circ\Xi_{1}^{-1}\circ \Xi_{1}\circ\phi_{J\nabla Q_{\e}}^t\circ \Xi_{1}^{-1} =\Xi_{2}\circ\Xi_{1}^{-1}\circ  \phi^t_{J\nabla\ti Q_{\e}}
$$
with $\ti Q_{\e}(u,v)=(Q_{\e}\circ\Xi_{1}^{-1})(u,v)=q_{\e}(v)$. Since $\phi^t_{J\nabla\ti Q_{\e}}(u,v)=(u-tq_{\e}'(v),v)$ and $\Xi_{2}\circ\Xi_{1}^{-1}(u,v)=(u-\ln v,v)$ we get (\ref{eq:3.11bis}).

\end{proof}

\section{Fundamental domains and first return maps}\label{sec:4}
We construct in this Section  adapted fundamental domains $\cF_{\e,y_{*}}$ for the maps $(f_{\e})_{\e}$ satisfying conditions \ref{iii0}, \ref{iii1},  \ref{iii2}, \ref{iii3} of Theorem  \ref{theo:2.1} and define their first return maps $\hat f_{\e}$ in $\cF_{\e,y_{*}}$. 

\subsection{Fundamental domains}\label{sec:4.1}
Let  $V_{}$ be the domain of Theorem  \ref{theo:2.1}. One can choose $x_{*}>0$   such that for any $\e$ small enough
$$(x_{*},0)\in V_{}\quad \textrm{and}\quad f_{\e}^{-1}(x_{*},0)\notin V_{}.$$ 
For $y_{*}>0$  small enough, we define the   vertical segment 
$$L_{x_{*},y_{*}}:=\{(x_{*},ty_{*}),\ 0<t<1\}$$ and the  domain $$\cF_{\e,x_{*},y_{*}}$$ as the interior of the contour defined by (see Figure \ref{fig:2})
\begin{description} \item[(a)] the segment $[f_{\e}(x_{*},0),(x_{*},0)]$, 
\item[(b)] the transversal  $\overline{L_{x_{*},y_{*}}}$
\item[(c)] the piece of hyperbola joining  $(x_{*},y_{*})$ to $f_{\e}(x_{*},y_{*})$ ({\it cf.} Remark \ref{rem:2.1}) 
\item[(d)] the curve  $f_{\e}(\overline{L_{x_{*},y_{*}}})$.
\end{description}
We  shall often drop the index $x_{*}$  in the notations of $L_{x_{*},y_{*}}$, $\cF_{\e,x_{*},y_{*}}$ and simply set
$$L_{y_{*}}:=L_{x_{*},y_{*}}\qquad \textrm{and}\qquad \cF_{\e,y_{*}}=\cF_{\e,x_{*},y_{*}}.$$
 If $y_{*}$ is small enough  one has $\cF_{\e,y_{*}},\overline{L_{y_{*}}}\subset V_{}$.
We set 
$$\ti \cF_{\e,y_{*}}=\cF_{\e,y_{*}}\cup L_{y_{*}}.$$
\subsection{First return maps}\label{sec:4.2}
Our aim in this subsection is to define  the first return map of $f_{\e}$ in $\ti \cF_{\e,y_{*}}$.

Since $\Sigma_{\e}$ is a separatrix for $f_{\e}$ we can define (see Remark \ref{rem:2.2})
$$N_{}(\e)\mathop{=}_{defin.}\min\{n\in\N^*,f_{\e}^{-n}(]f_{\e}(x_{*}),x_{*}])\subset V_{}\}.$$
We note that if $\e$ is small enough $N_{}(\e)$ is independent of $\e$, so we shall denote it by $N_{}$. Moreover, if
 $\e$ and $y_{*}$ are small enough
\be N_{}\mathop{=}_{defin.}\min\{n\in\N^*,f_{\e}^{-n}(\ti \cF_{\e,y_{*}})\subset V_{}\}.\label{defN}\ee
\begin{lemma}There exists a constant $0<c_{*}<1$ such that for $(x,y)\in\cF_{\e,c_{*}y_{*}}$
\be \ti n_{\e}(x,y):\mathop{=}_{defin.}\min\{j\in\N^*, f_{\e}^j(x,y)\in f_{\e}^{-N_{}}(\ti\cF_{\e,y_{*}})\}<\infty. \label{tildenante*}\ee
One has 
\be \ti n_{\e}(x,y)\asymp \ln(xy)/\l.\label{tildenante}\ee
\end{lemma}
\begin{proof}Note that the domain $f_{\e}^{-N_{}}(\cF_{\e,y_{*}})\subset \ti V_{}$ is the interior of the contour defined by:
 \begin{description} \item[(a)]  the segment $[f_{\e}^{-(N_{}-1)}(x_{*},0),f_{\e}^{-N_{}}(x_{*},0)]\subset W^{u}_{f_{\e}}(o)\cap V_{}\subset \{0\}\times \R$
\item[(b)] the curve  $f_{\e}^{-N_{}}(\overline{L_{y_{*}}})$
\item[(c)] a curve joining  $f_{\e}^{-N_{}}(x_{*},y_{*})$ to $f_{\e}^{-(N_{}-1)}(x_{*},y_{*})$ 
\item[(d)] the curve $f_{\e}^{-(N_{}-1)}(\overline{L_{y_{*}}})$
\end{description} 
and
$$f_{\e}^{-N_{}}(\ti \cF_{\e,y_{*}})=f_{\e}^{-N_{}}(\cF_{\e,y_{*}})\cup f_{\e}^{-N_{}}(L_{y_{*}}).
$$
  Note that the lines $f_{\e}^{-N_{}}(\overline{L_{y_{*}}})$, $f_{\e}^{-(N_{}-1)}(\overline{L_{y_{*}}})$ are transversal to the segment $[f_{\e}^{-(N_{}-1)}(x_{*},0),f_{\e}^{-N_{}}(x_{*},0)]$. 
  
  Now let $(x,y)\in \ti \cF_{\e,y_{*}}$.  We denote by $\cH_{x,y}$ the hyperbola  $$\cH_{x,y}:=\{ (x',y'),\  x'y'=xy\}$$ and if $z,z'\in\cH_{x,y}$, by $\cH_{x,y}(z,z')$ the arc of hyperbola of $\cH_{x,y}$ between $z$ and $z'$ which is open in $z$ and closed in $z'$. If $y>0$ is small enough, $\cH_{x,y}$  intersects $f_{\e}^{-N_{}}(\ti \cF_{\e,y_{*}})\subset V_{}$ in an arc of hyperbola  of the form $\cH_{x,y}(p,f_{\e}^{-1}(p))$ with  $p\in f_{\e}^{-(N_{}-1)}({L_{y_{*}}})$ and $f_{\e}^{-1}(p)\in  f_{\e}^{-N_{}}({L_{y_{*}}})$. The sets $f^{-j}_{\e}(\cH_{x,y}(p,f_{\e}^{-1}(p))$, $j\geq 0$ form a partition of the semi-arc of parabola $\bigcup_{n\geq 0}\cH_{x,y}(p,f_{\e,k}^{-n}(p)$ which contains $(x,y)$. In particular, there exists $j\geq 0$ (in fact $j\geq 1$) such that  $(x,y)\in f^{-j}_{\e}(\cH_{x,y}(p,f_{\e}^{-1}(p))$ or equivalently 
$$f_{\e}^j(x,y)\in \cH_{x,y}(p,f_{\e}^{-1}(p))\subset f^{-N_{}}_{\e}(\ti \cF_{\e,y_{*}}).$$
This proves (\ref{tildenante*}).

\medskip
To prove (\ref{tildenante}) we note that there exists an interval $I_{2}$ not containing 0 and depending only on $x_{*},y_{*}$ such that $f_{\e}^{-N_{}}(\ti \cF_{\e,y_{*}})\subset \R\times I_{2}$. We then use Lemma \ref{prop:linearization} and the fact that $|\bar t_{I_{2}}(x,y)-\ti n_{\e}(x,y)|\leq 1$.
\end{proof}
We now define
\be n_{\e}=N_{}+\ti n_{\e}.\label{defneps}\ee
By (\ref{tildenante}) one has 
\be n_{\e}(x,y)\asymp \ln(xy)/\l.\label{tilden}\ee
The map  
$\hat f_{\e}:\ti \cF_{\e,c_{*}y_{*}}\to \ti \cF_{\e,y_{*}}$  defined by 
\be \hat f_{\e}=f_{\e}^{n_{\e}}\label{deffhat}\ee
 is the first return map of $f_{\e}$ in $\ti  \cF_{\e,y_{*}}$ (for points starting in $ \ti \cF_{\e,c_{*}y_{*}}$). Note that $\hat f_{\e}$ is {\it not} $C^k$ on the whole domain $\ti \cF_{\e,c_{*}y_{*}}$.

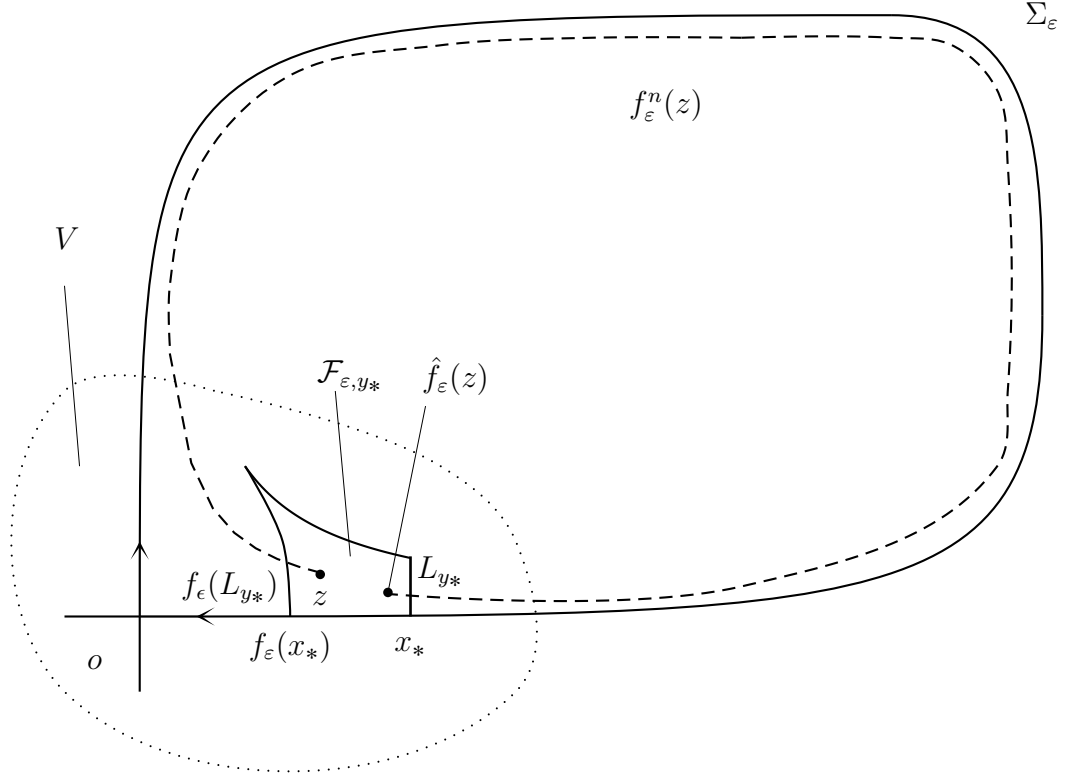
\begin{figure}
\begin{center}
\begin{pspicture}(-1.5,-1.7)(11,10)
\psset{xunit=2cm,yunit=2cm}
\def\Hypa{.7 x div}
\def\Hypb{.7 x div 0.5 mul}
\psplot{.7}{1.8}{\Hypa}
\pscurve[linestyle=dashed](0.2,1.8)(0.2,2.2)(0.4,3)(1,3.6)(2,3.8)(3,3.85)(4,3.85)
\pscurve[linestyle=dashed](4,3.85)(5,3.85)(5.4,3.8)(5.7,3.5)(5.77,3)(5.77,1.5)(5.7,1)(4,.2)(1.65,0.15)
\pscurve[linestyle=dotted](0.5,1.5)(-0.5,1.5)(-0.5,-0.5)(2.5,-0.5)(2.5,0.5)(2,1)(0.5,1.5)
\rput(6,4){$\Sigma_{\e}$}
\rput(3.5,3.4){$f_{\e}^n(z)$}
\rput(-0.5,2.5){$V_{}$}
\psline[linewidth=0.05pt](-0.5, 2.2)(-0.4,1)
\rput(1.4,1.6){$\cF_{\e,y_{*}}$}
\psline[linewidth=0.05pt](1.3,1.5)(1.4,0.4)
\psplot[linestyle=dashed,plotpoints=8,dotsize=2pt]{.2}{1.2}{\Hypb}
\psdot(1.2,0.28)
\rput(1.2,0.12){$z$}

\rput(1.8,-0.2){$x_{*}$}
\rput(1,-0.2){$f_{\e}(x_{*})$}

\psdot(1.65,0.16)
\rput(2.1,1.6){$\hat f_{\e}(z)$}
\psline[linewidth=0.2pt](1.9,1.4)(1.65,0.16)

\psbezier(0,0)(0,4)(0,4)(5,4)
\psbezier(0,0)(5,0)(6,0)(6,2)
\psbezier(6,2)(6,3)(6,4)(5,4)

\psline[linewidth=0.5pt, arrowsize=2pt 10, arrowlength=1,arrowinset=0.7]{>-}(0.5,0)(0,0)
\psline(-0.5,0)(0,0)
\psline[linewidth=0.5pt, arrowsize=2pt 10, arrowlength=1,arrowinset=0.7]{->}(0,0)(0,0.5)
\psline[linewidth=1pt](1.8,0)(1.8,0.4)
\psline(0,-0.5)(0,0)
\pscurve(1,0)(.95,0.5)(.7,1)
\rput(2,0.3){$L_{y_{*}}$}
\rput(.6,.2){$f_{\epsilon}(L_{y_{*}})$}
\rput(-0.3,-0.3){$o$}
\end{pspicture}
\end{center}
\caption{Fundamental domain $\cF_{\e,y_{*}}$ for $f_{\e}$ and the first return map $\hat f_{\e}$.}\label{fig:2}
\end{figure}

\subsection{Estimates on first return maps}
We denote for $a\in\R$
$$T_{a}:(u,v)\mapsto (u+a,v)$$
and we recall the definition (\ref{defXi}) of the symplectic diffeomorphisms $\Xi_{1},\Xi_{2}$.

We observe that there exist open sets $W_{1}\subset\R^*_{+}\cap \R$, $W_{2}\subset \R\times\R_{+}^*$ such that for any  $\e$ and $y_{*}>0$  small enough 
$$\ti \cF_{\e,y_{*}}\subset W_{1}\subset V_{},\qquad f_{0}^{-N_{}}(\ti \cF_{\e,y_{*}})\subset W_{2}\subset V_{}.$$
\begin{lem}\label{lemma:4.3}There exists a $C^k$ function $\sigma_{0,N_{}}\in C^k(\R^*_{+},\R)$ such that on $\Xi_{2}(W_{2})$ one has 
\be \Xi_{1}\circ f_{0}^{N_{}}\circ \Xi_{2}^{-1}=T_{\sigma_{0,N_{}}}.\label{eq:4.6}
\ee
\end{lem}
\begin{proof} From condition  \ref{iii3} one can write on $\R^2$
$$f_{0}=\phi^1_{J\nabla H_{0}}$$
hence 
$$f_{0}^{N_{}}=\phi^{N_{}}_{J\nabla H_{0}}, \qquad\textrm{where}\  H_{0}\ |_{V_{}}=Q_{0}.$$
If $(u,v)\in \Sigma_{2}(W_{2})$ and $(\ti u,\ti v)=\Xi_{1}(f_{0}^{N_{}}(\Xi_{2}^{-1}(u,v)))$ one then  has
$$Q_{0}(\Xi_{1}^{-1}(\ti u,\ti v))=Q_{0}(f_{0}^{N_{}}(\Xi_{2}^{-1}(u,v)))=Q_{0}(\Xi_{2}^{-1}(u,v))$$
hence
$$q_{0}(\ti v)=q_{0}(v)$$
and thus $\ti v=v$. Since the map $(u,v)\mapsto (\ti u, \ti v)$ is symplectic, this forces $\ti u=u+\ti \s_{0,N_{}}(v)$ for some $C^k$ function $\ti\s_{0,N_{}}$;  this function  can be extended as a $C^k$ function $\s_{0,N_{}}:\R\to\R$.

\end{proof}

Recall the definition (\ref{deffhat}) of $\hat f_{\e}$.
\begin{lemma} There exists
   a continuous family $(\hat \eta_{\e})_{\e}$ of $C^k$ symplectic   diffeomorphisms defined on $\R^2$ and   a neighborhood $W$ of $ f_{\e}^{-1}(\ti \cF_{\e,y_{*}})\cup \ti \cF_{\e,y_{*}}\cup f_{\e}(\ti \cF_{\e,y_{*}})$ such that
 \be\begin{cases}&\lim_{\e\to 0} \|\hat \eta_{\e}-id\|_{k}=0\\
 &\hat\eta_{\e}(W\cap(\R\times\{0\})\subset \R\times\{0\}
\end{cases} \label{hateta}
\ee
 and on a neighborhood of $\ti\cF_{\e,c_{*}y_{*}}$ one has 
\be \Xi_{1}\circ  \hat f_{\e}\circ  \Xi_{1}^{-1}=\hat \eta_{\e}\circ T_{\hat l_{\e}}\label{eq:5.18}\ee  
where
\be\hat l_{\e}(v)=\sigma_{0,N_{}}(v)+\hat n_{\e}(u,v)q_{\e}'(v)-\ln v,\qquad \textrm{with}\quad \hat n_{\e}=\ti n_{\e}\circ \Xi_{1}^{-1}.\label{eq:5.18'}\ee
\end{lemma}
\begin{proof} We write (we use (\ref{deffhat}), (\ref{defneps}), \ref{iii2})
\begin{align*} \hat f_{\e}&=f_{\e}^{N_{}+\ti n_{\e}}\\
&=f_{\e}^{N_{}}\circ \phi_{J\nabla Q_{\e}}^{\ti n_{\e}}\\
&=\eta_{\e}\circ  f_{0}^{N_{}}\circ \phi_{J\nabla Q_{\e}}^{\ti n_{\e}}\end{align*}
with \be \eta_{\e}=f_{\e}^{N_{}}\circ f_{0}^{-N_{}}.\label{eq:4.21}\ee
As a consequence, if we set
\be \hat \eta_{\e}=\Xi_{1}\circ \eta_{\e}\circ \Xi_{1}^{-1}\qquad  \textrm{and}\qquad \hat n_{\e}=\ti n_{\e}\circ \Xi_{1}^{-1},\label{eq:4.22}
\ee
we have using (\ref{eq:4.6})
\begin{align*}\Xi_{1}\circ \hat f_{\e}\circ \Xi_{1}^{-1}&=\hat\eta_{\e}\circ (\Xi_{1}\circ f_{0}^{N_{}}\circ\Xi_{2}^{-1})
\circ (\Xi_{2}\circ \phi^{\ti n_{\e}}_{J\nabla Q_{\e}} \circ \Xi_{2}^{-1}) \circ (\Xi_{2}\circ\Xi_{1}^{-1}) \\
&=\hat \eta_{\e}\circ T_{\s_{0,N_{}}}\circ \phi^{\ti n_{\e}\circ  \Xi_{1}^{-1}}_{J\nabla (Q_{\e}\circ \Xi_{2}^{-1})} \circ T_{-\ln v}\\
&=\hat \eta_{\e}\circ T_{\s_{0,N_{}}}\circ T^{\hat n_{\e}}_{q'_{\e}}\circ T_{-\ln v}\\
&=\hat \eta_{\e}\circ T_{\s_{0,N_{}}+\hat n_{\e}q'_{\e}-\ln v}
\end{align*}
which is (\ref{eq:5.18}) together with (\ref{eq:5.18'}).

Note that by (\ref{eq:4.21}), (\ref{eq:4.22}), Remark \ref{rem:2.2} and the fact that $\R\ni\e\mapsto f_{\e}\in C^k(V_{k},\R^2)$ is continuous, one has 
$$\begin{cases}&\lim_{\e\to 0}\|\hat \eta_{\e}-id\|_{C^k}=0\\
&\hat\eta_{\e}(W\cap (\R\times\{0\}))\subset \R\times\{0\}.
\end{cases}$$

\end{proof}

\section{Renormalization}\label{sec:5}
We define in this section a {\it renormalization} $\bar f_{\e}$ of the map $f_{\e}$. The first return map $\hat f_{\e}$ of $f_{\e}$ in the fundamental domain $\cF_{\e,y_{*}}$  we have constructed in the previous Section \ref{sec:4} is not differentiable at every point (see (\ref{eq:5.18}), (\ref{eq:5.18'} and the fact that the integer valued function $\hat n_{\e}$ has  in general discontinuity points). On the other hand, if one {\it glues} the ``vertical'' boundaries of $\cF_{\e,y_{*}}$ by $f_{\e}$ we obtain an {\it abstract} open annulus $\ti F_{\e,y_{*}}/f_{\e}$ (see Subsections \ref{sec:glueing}, \ref{sec:5.2}) and the map $\hat f_{\e}$ is now $C^k$ on it.  We can {\it uniformize}\footnote{Uniformizing the annulus is equivalent to  conjugating $f_{\e}$ to $(x,y)\mapsto (x+1,y)$ on a domain containing $\cF_{\e,y_{*}}$. This procedure is, in a different context, the one described in \cite{Y}. We shall often call the uniformization operation  {\it normalization} in reference to the corresponding renormalization procedure defined for quasi-periodic cocycles, {\it cf.} \cite{K}, \cite{AK1}. } this abstract annulus so that it becomes the standard (with the usual topology)  open annulus $\R/\Z\times ]0,c[$ (some $c>0$), see Subsection \ref{sec:5.3}, and the map $\hat f_{\e}$ in these new coordinates turns into  a $C^k$ diffeomorphism  $\bar f_{\e}$ defined on (part of) this standard annulus. This is the\footnote{One should say ``a'' instead of ``the'' since the uniformizing/normalizing procedure is not unique.} {\it renormalized} diffeomorphism associated to $f_{\e}$.  
\subsection{Glueing}\label{sec:glueing}
Let $\cF$ be an open set of $\R^2$, $L$ a 1-dimensional submanifold of $\R^2$ and $f$ an orientation preserving  smooth diffeomorphism from a neighborhood of $\cF\cup L$ to a neighborhood of $f(\cF\cup L)$. We assume that:
\begin{enumerate}
\item\label{i1} $f(\cF\cup L)\cap (\cF\cup L)=\emptyset$;
\item\label{i2} $\cF\cup L$ is a 2-dimensional submanifold of $\R^2$ with boundary  and this boundary is  $\pa(\cF\cup L)=L$; in particular, for any point $p\in L$ there exists an open set $U_{p}$, $p\in U_{p}\subset\R^2$, and a smooth diffeomorphism $\ph_{p}:U_{p}\to \ph_{p}(U_{p})\subset\R^2$ such that  $\ph_{p}(U_{p}\cap L)=\ph_{p}(U_{p})\cap (\R\times\{0\})$ and $\ph_{p}(U_{p}\cap \cF)=\ph_{p}(U_{p})\cap (\R\times \R^*_{+})$;
\item\label{item4} for any $p\in\cF\cup L$ and $U_{p}$, one has $U_{p}\cap f(\cF\cup L)=\emptyset$;
\item\label{i3} for any $p\in L$ one has $f^{-1}(f(U_{p})\cap \cF)=\ph_{p}^{-1}(\ph_{p}(U_{p})\cap (\R\times \R^*_{-}))$, for any of the previous chart $(U_{p},\ph_{p})$ at $p$.
\end{enumerate}
We define the {\it topological space} $(\cF\cup L,\cT)$ as being the set  $\cF\cup L$ endowed with the following topology $\cT$: a subset  $S$ of $\cF\cup L$ is an element of $\cT$  (i.e. an open set)  if for every $p\in S$ there exists an open set  $V\subset \R^2$ (contained in a neighborhood of $\cF\cup L$ where $f$ is defined)  such that  $V\cap f(\cF\cup L)=\emptyset$ and $p\in (V\cup f(V))\cap(\cF\cup L)\subset S$.

We can then define the following {\it differentiable structure} on $(\cF\cup L,\cT)$ as follows; (a):  if $p\in\cF$  we define the  local chart    $C_{p}:=(W_{p},id)$ where $W_{p}$ is an open set of $\R^2$ such that $p\in W_{p}\subset \cF$;  and (b): if $p\in L$ we define the local chart $C_{p}:=(W_{p},\psi_{p})$ where  $W_{p}$ is the open set of $\cF\cup L$ (see condition (\ref{item4}))  $W_{p}=(\cF\cup L)\cap (U_{p}\cup f(U_{p}))$ (here $(U_{p},\ph_{p})$ is the  local chart for $p\in L$ as defined in  (\ref{i2})) and where $\psi_{p}$ is defined by  (we use condition (\ref{i3})),
$$\begin{cases}&\psi_{p}=\ph_{p}\quad \textrm{on} \quad U_{p}\cap (\cF\cup L)=\ph_{p}^{-1}((\ph_{p}(U_{p})\cap (\R\times\R_{+}))\\
&\psi_{p}=\ph_{p}\circ f^{-1}\quad \textrm{on} \quad f(U_{p})\cap \cF=f\circ\ph_{p}^{-1}((\ph_{p}(U_{p})\cap (\R\times\R_{-}^*)).
\end{cases}
$$
We denote by $\cA$ the collection of all these local charts $C_{p}$ and we set $(\cF\cup L)/f=(\cF\cup L,\cT,\cA)$. 

\begin{rem}\label{rem:5.4.0}If we  assume in addition that $f$ preserves the standard symplectic form $dx\wedge dy$ on $\R^2$, we can endow $(\cF\cup L)/f$ with a  symplectic form $\omega$. 
\end{rem}

\begin{rem}\label{rem:5.4.1}
If $g:\cF\to g(\cF)$ is a smooth  diffeomorphism defined in a neighborhood of $\cF$, it induces a smooth diffeomorphism (that we still denote $g$) $g:(\cF\cup L)/f\to (g(\cF)\cup g(L))/(g\circ f\circ g^{-1})$.
\end{rem}
\begin{rem}\label{rem:5.4.2}If $\cF=[0,1[\times ]0,1[$, $L=]0,1[$ and $f=T_{1}:(x,y)\mapsto (x+1,y)$ one sees that  $(\cF\cup L)/T_{1}$ is (diffeomorphic to) the standard open annulus $(\R/\Z\times ]0,1[,{\rm can.})$ endowed with its canonical differentiable structure. 
\end{rem}

\subsection{The space $(\cF_{\e,y_{*}}\cup L_{y_{*}})/f_{\e}$}\label{sec:5.2} If $\e$ and $y_{*}$ are small enough item (\ref{i1}) is satisfied and we can find charts $(p,U_{p})$ such that items (\ref{i2}), (\ref{item4}) (\ref{i3}) are  satisfied. See Figure \ref{fig:3}.  We can then define the manifold $(\cF_{\e,y_{*}}\cup L_{y_{*}})/f_{\e}$. We shall see that it is an annulus without boundary, {\it cf.} Lemma \ref{lemma:5.2}.

Note that if $0<c_{*}<1$, the smaller  set $\ti \cF_{\e,c_{*}y_{*}}=\cF_{\e,c_{*}y_{*}}\cup L_{c_{*}y_{*}}$ is an open subset of $(\cF_{\e,y_{*}}\cup L_{y_{*}})/f_{\e}$ (which means that it belongs to $\cT$) and it can be endowed with the topology and differentiable structure induced by the inclusion. We denote $(\cF_{\e,c_{*}y_{*}}\cup L_{c_{*}y_{*}})/f_{\e}$  the thus obtained submanifold of $(\cF_{\e,y_{*}}\cup L_{y_{*}})/f_{\e}$.  The following lemma is then tautological
\begin{lemma}The map $\hat f_{\e}$ induces a $C^k$ map $\ti \cF_{\e,c_{*}y_{*}}/f_{\e}\to \ti \cF_{\e,y_{*}}/f_{\e}$.
\end{lemma}
We shall need in Section \ref{sec:6} the following lemma:

\begin{lemma}\label{volform}There exists a probability measure with positive density  $\pi_{\e,y_{*}}$   on $\ti \cF_{\e,y_{*}}/f_{\e}$ which is $\hat f_{\e}$ invariant: for any measurable set  $A\in \ti \cF_{\e,y_{*}}/f_{\e}$ such that $\hat f_{\e}^{-1} (A)\in \ti \cF_{\e,y_{*}}/f_{\e}$ one has $\pi_{\e,y_{*}}(A)=\pi_{\e,y_{*}}(\hat f_{\e}^{-1}(A))$.
\end{lemma}
\begin{proof}We shall in fact construct this measure  $\pi_{\e,y_{*}}$  on the bigger set  $\hat  \cF_{\e,y_{*}}/f_{\e}$
$$\hat \cF_{\e,y_{*}}=\ti  \cF_{\e,y_{*}}\cup \sigma(\ti  \cF_{\e,y_{*}})$$
where $\sigma :\R^2\to\R^2$ is the reflection $(x,y)\mapsto (x,-y)$ (it commutes with $f_{\e}$ in $V_{}$, see condition \ref{iii2}).
From Remark \ref{rem:5.4.0} there exists a symplectic form $\omega_{\e}$ on $\hat  \cF_{\e,y_{*}}/f_{\e}$.
Note that the first return map $\hat f_{\e}$ is not defined on the whole  set $\hat \cF_{\e,y_{*}}/f_{\e}$ but nevertheless
$$(\hat f_{\e})^*\omega_{\e}=\omega_{\e}$$
whenever this formula makes sense.
The probability measure  $\pi_{\e,y_{*}}$ defined by
$$\pi_{\e,y_{*}}(A)=\int_{A}|\omega_{\e}|/\int_{\cF_{\e,y_{*}}}|\omega_{\e}|$$
is $\hat f_{\e}$ invariant.
\end{proof}

\begin{figure}
\begin{center}
\begin{pspicture}(-1,-1)(10,5)
\psset{xunit=1.7cm,yunit=1.2cm}
\psline(2,0)(4,0)
\psline(4,0)(4,3.7)

\psbezier(2,0)(1.9,2)(1.9,2)(1,3.5)
\pscurve[linestyle=dotted](1,3.5)(1.5,3.7)(2,3.6)(2.5,3.7)(3,3.6)(3.5,3.7)(4,3.7)
\psarc[linestyle=dotted](4,2){0.8}{-90}{90}
\psarc(4,2){0.8}{90}{-90}
\psdot(4,2)
\rput(4.2,2.2){$x$}
\psdot(1.6,2.46)
\rput(1.3,2.2){$f_{\e}(x)$}
\rput(3.4,1.2){$U_{-}$}
\psline[linewidth=0.1pt](3.5,1.27)(3.8,1.6)
\rput(4.6,1.2){$U_{+}$}
\psline[linewidth=0.1pt](4.5,1.27)(4.2,1.6)
\rput(2.4,1.7){$f_{\e}(U_{+})$}
\psline[linewidth=0.1pt](2.4,1.9)(2,2.4)
\rput(4.4,0.4){$L_{y_{*}}$}
\rput(1.4,0.4){$f_{\e}(L_{y_{*}})$}
\rput(3,-.3){$[x_{*},f_{\e}(x_{*})]\times\{0\}$}
\rput(4.7,3.5){$\cF_{\e,y_{*}}$}
\psline[linewidth=0.1pt](4.3,3.4)(3.6,3.2)

\pscurve(1.3,3)(1.8,3.2)(2.2,2.8)(2,2.2)(1.8,2)
\pscurve[linestyle=dotted](1.8,2)(1.4,2)(1.1,2.1)(1.1,2.6)(1.3,3)

\pscustom[fillstyle=hlines,linestyle=none]{%
\pscurve[linewidth=0pt,linecolor=blue,liftpen=0](1.3,3)(1.8,3.2)(2.2,2.8)(2,2.2)(1.8,2)
\pscurve[linewidth=0pt,liftpen=0](1.8,2)(1.6,2.45)(1.4,2.8)(1.3,3)}

\pscustom[fillstyle=hlines,linestyle=none]{%
\psarc[linewidth=0pt,linecolor=blue,liftpen=0](4,2){0.8}{90}{-90}
\psline[linewidth=0pt,liftpen=0](4,2.65)(4,2.35) }

\end{pspicture}
\end{center}
\caption{Glueing: $(\cF_{\e,y_{*}}\cup L_{y_{*}})/f_{\e}$ }\label{fig:3}
\end{figure}
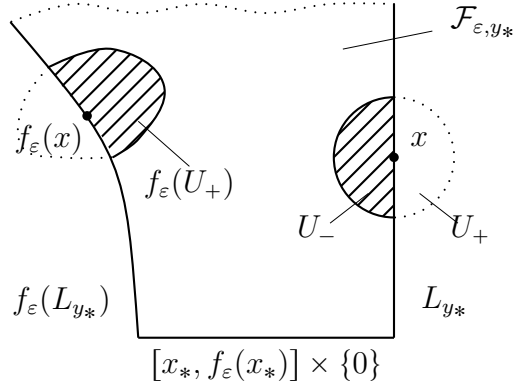

\subsection{Normalization of $f_{\e}$}\label{sec:5.3} 
We now uniformize the abstract annulus $\ti \cF_{\e,y_{*}}/f_{\e}$. To do that it is enough to normalize $f_{\e}$ in the sense of Item  \ref{il5.2.2}  of the following lemma.

\begin{lemma}\label{lemma:5.2}[Normalization Lemma] There exists a continuous family  $(h_{\e})_{\e}$ of (not necessarily symplectic)  $C^{k}$-diffeomorphisms  defined on a neighborhood of $\ti \cF_{\e,y_{*}}$ such that for some $c_{}>0$
\begin{enumerate}
\item $h_{\e}$ sends $\ti\cF_{\e,y_{*}}/f_{\e}$ to the standard open annulus $( (\R/\Z)\times ]0,c_{}[,{\rm can.})$
\item\label{il5.2.2} $h_{\e}\circ f_{\e}\circ h_{\e}^{-1}=T_{1}:(x,y)\mapsto (x+1,y)$
\item\label{il5.2.3} $h_{\e}([x_{*},f_{\e}(x_{*)}] \times\{0\}=[0,1[\times \{0\}$.
\end{enumerate}
\end{lemma}
\begin{proof}Using  condition \ref{iii2} and the change of coordinates (\ref{defXi}) of Section \ref{sec:3}, we see  that on a neighborhood of $\ti\cF_{\e,y_{*}}$ one has (we use the notation $(x,y)$ for $(u,v)$)
$$\Xi_{1}\circ f_{\e}\circ \Xi_{1}^{-1}=T_{q'_{\e}} :(x,y)\mapsto (x+q'_{\e}(y),y).$$
If $g_{\e}$ is the (not  necessarily symplectic) smooth diffeomorphism
\be g_{\e}:(x,y)\mapsto (\frac{x}{q'_{\e}(y)},y)\label{e5.23}\ee
one has 
\be g_{\e}\circ \Xi_{1}\circ f_{\e}\circ \Xi_{1}^{-1}\circ g_{\e}^{-1}=T_{1}.\label{e5.24}\ee
The set $\overline{(g_{\e}\circ \Xi_{1})(\cF_{\e,y_{*}})}$  is of the form
$$\overline{(g_{\e}\circ\Xi_{1})(\cF_{\e,y_{*}})}=\{(x,y),\ y\in [0,c_{}],\ \g_{\e}(y)\leq x\leq \g_{\e}(y)+1\}$$ where $c_{}>0$, $\g_{\e}:[0,c_{}]\to \R_{+}$ is $C^k$,
$\g_{\e}(0)=0$ and the map $\R\ni \e\mapsto \g_{\e}\in C^k([0,c_{}],\R)$ is continuous. This indeed  follows from the definition of   $\ti \cF_{\e,y_{*}}$ in subsection \ref{sec:4.1}), the definition of $\Xi_{1}$ (\ref{defXi}) and (\ref{e5.23}), (\ref{e5.24}). As a consequence, if we denote 
\be j_{\e}:(x,y)\mapsto (x-\gamma_{\e}(y),y)\label{e5.25}\ee
we have \be\begin{cases}&j_{\e}\circ T_{1}=T_{1}\circ j_{\e}\\
&j_{\e}\biggl( (g_{\e}\circ\Sigma_{1})(\cF_{\e,y_{*}})\biggr)=]0,1[\times ]0,c_{}[\\
&j_{\e}\biggl((g_{\e}\circ\Sigma_{1})(L_{y_{*}})\biggr)=\{0\} \times ]0,c_{}[.
\end{cases}\label{e5.26}
\ee
By Remarks \ref{rem:5.4.1}, \ref{rem:5.4.2}, the map 
\be h_{\e}=j_{\e}\circ g_{\e}\circ \Xi_{1}\label{e5.26}\ee
is a diffeomorphism that sends $\ti \cF_{\e,y_{*}}/f_{\e}$ to the standard annulus $([0,1[\times ]0,c_{}[)/T_{1}\simeq (\R/\Z)\times ]0,c_{}[$ and such that 
$$h_{\e}\circ f_{\e}\circ h_{\e}^{-1}=T_{1}.$$

To conclude the proof, we notice (\ref{il5.2.2}) is an immediate  consequence of the definition  (\ref{e5.26}) of $h_{\e}$.

\end{proof}

\begin{rem}\label{rem:lemma5.2}Note that if $T_{a}(x,y)=(x+a(y),y)$ one has
$$(h_{\e}\circ\Xi_{1}^{-1})\circ T_{a}\circ (h_{\e}\circ \Xi_{1}^{-1})^{-1}=T_{\ti a},\qquad \ti a(y)=a(y)/q_{\e}'(y).
$$
\end{rem}

\subsection{The renormalization $\bar f_{\e}$ of $f_{\e}$}\label{sec:5.4}

There exists $\d_{}\in ]0,c_{}[$ such that 
the map 
\be \bar f_{\e}\mathop{=}_{defin.}h_{\e}\circ \hat{f}_{\e}\circ h_{\e}^{-1}:\R/\Z\times ]0,\d_{}[\to \R/\Z\times ]0,c_{}[.\label{ee5.28}\ee
is well defined and is a $C^k$ diffeomorphism onto its image.

\begin{prop}\label{prop:main}
One has 
\be \bar f_{\e}=\bar\eta_{\e}\circ T_{l_{\e}}\label{eq:prop:5.4}\ee where $\bar\eta_{\e}$ is a $C^k$ diffeomorphism defined on $\R/\Z\times]0,\d_{}[$ and  $l_{\e}\in C^{k}(]0,c_{}[, \R/\Z)$; they satisfy
\be l_{\e}(y)=\frac{\sigma_{0,N_{}}(y)}{q'_{\e}(y)}-\frac{\ln y}{q'_{\e}(y)} \quad \mod\ \Z\label{5.27}\ee
\be\lim_{\e\to0}\|\bar\eta_{\e}-id\|_{C^{k}}=0\label{5.28}\ee
\be \bar\eta_{\e}:(x,y)\mapsto (x+a_{\e}(x,y),y+yb_{\e}(x,y))\label{labelab} \ee
where $a_{\e}\in C^k$, $b_{\e}\in C^{k-1}$ are  functions defined on $\R/\Z\times (0,\d_{})$.

Moreover, the map $\bar f_{\e}$ preserves a probability measure $\bar \pi_{\e,y_{*}}$ with positive density defined on $\R/\Z\times ]0,c_{}[$.
\end{prop}
\begin{proof} 
By (\ref{eq:5.18}) and Remark \ref{rem:lemma5.2} after Lemma \ref{lemma:5.2}
\begin{align*}\bar f_{\e}&=(h_{\e}\circ\Xi_{1}^{-1})\circ \hat \eta_{\e}\circ (h_{\e}\circ\Xi_{1}^{-1})^{-1}\circ (h_{\e}\circ \Xi_{1}^{-1})\circ T_{\hat l_{\e}}\circ (h_{\e}\circ \Xi_{1}^{-1})^{-1}\\
&=\bar\eta_{\e}\circ T_{l_{\e}}
\end{align*}
where 
\be\bar\eta_{\e}=(h_{\e}\circ \Xi_{1}^{-1})\circ \hat \eta_{\e}\circ (h_{\e}\circ \Xi_{1}^{-1})^{-1}\qquad\textrm{and}\qquad  l_{\e}(y)=(1/q'_{\e}(y))\hat l_{\e}(y).\label{5.31}\ee
Since $\bar \eta_{\e}:=(h_{\e}\circ \Xi_{1}^{-1})\circ \hat \eta_{\e}\circ (h_{\e}\circ \Xi_{1}^{-1})^{-1}$ and $\bar f_{\e}$ are $C^{k}$, the function ${l_{\e}}:]0,c_{}[\to\R/\Z$ is also $C^{k}$  and 
$$l_{\e}(y)=(1/q'_{\e}(y))\hat l_{\e}(y).$$
 By (\ref{eq:5.18'}) (remember that $\hat n_{\e}$ takes its value in $\Z$)
\begin{align*}l_{\e}(y)&= \frac{\sigma_{0,N_{}}(y)}{q'_{\e}(y)}+\hat n_{\e}(x,y)-\frac{\ln y}{q'_{\e}(y)}\\
&=\frac{\sigma_{0,N_{}}(y)}{q'_{\e}(y)}-\frac{\ln y}{q'_{\e}(y)} \quad \mod\ \Z\end{align*}
which is (\ref{5.27}).

Equation (\ref{5.28}) is a consequence of the definition of $\bar \eta_{\e}$, {\it cf.} (\ref{5.31}), the first equation of  (\ref{hateta}) and of the fact that $\R\ni \e\mapsto h_{\e}\in C^k$ is continuous (Lemma \ref{lemma:5.2}).

\medskip
We now claim that if $\bar \eta_{\e}(x,y)=(x+a_{\e}(x,y), y+\bar b_{\e}(x,y))$ one has for any $y$,
 \be \bar b_{\e}(x,0)=0.\label{barb}\ee 
 Indeed, since  
$$\bar \eta_{\e}:=(h_{\e}\circ \Xi_{1}^{-1})\circ \hat \eta_{\e}\circ (h_{\e}\circ \Xi_{1}^{-1})^{-1},$$
equality  (\ref{barb}) is a consequence of the second equation of  (\ref{hateta}), of item (\ref{il5.2.3}) of Lemma \ref{lemma:5.2} and of the fact that $\Xi_{1}(\R_{+}^*\times\{0\})=\R_{+}^*\times\{0\}$.

To prove   (\ref{labelab}) we thus notice  that  equality (\ref{barb}) gives us for  $\bar b_{\e}$ a decomposition $$\begin{cases}&\bar b_{\e}(x,y)=yb_{\e}(x,y)\\
&b_{\e}\in C^{k-1} .\end{cases}$$

Finally to conclude the proof of the Proposition
we observe that since the map $\hat f_{\e}:\ti\cF_{\e,c_{*}y_{*}}/f_{\e}\to \ti\cF_{\e,y_{*}}/f_{\e}$ preserves  the probability measure $\pi_{\e,y_{*}}$, {\it cf.} Lemma \ref{volform}, the diffeomorphism $\bar f_{\e}:\R/\Z\times ]0,\d_{}[\to \R/\Z\times ]0,c_{}[$ preserves the probability measure $\bar \pi_{\e,y_{*}}=(h_{\e})_{*}\pi_{\e,y_{*}}$ defined on $\R/\Z\times ]0,c_{}[$ (in the sense that if $A\subset \R/\Z\times ]0,c_{}[$ is a Borelian set such that $\bar f_{\e}^{-1}(A)\subset \R/\Z\times ]0,c_{}[$, one has $\bar \pi_{\e,y_{*}}(A)=\bar \pi_{\e,y_{*}}(\bar f_{\e}^{-1}(A))$).

\end{proof}

\section{Applying the Translated Curve Theorem}\label{sec:6}
We apply in this Section R\"ussmann's (or Moser's) Translated Curve Theorem to some {\it rescaled} version $\mathring{f}_{\e,n}$ of  the renormalization $\bar f_{\e}$ of $f_{\e}$ defined in Section \ref{sec:5.4}.
\subsection{The Translated  Curve Theorem}
Let  $\psi: \R/\Z\times ]e^{-1},1[\to\R/\Z\times\R$ ($\ln e=1$) be  a $C^k$  diffeomorphism defined on the annulus (or cylinder)  $\R/\Z\times ]e^{-1},1[$. We say that the graph ${\rm Gr}_{\g}:=\{(x,\gamma(x)): x\in\R/\Z\}$ of a continuous map $\gamma:\R/\Z\to \R/\Z\times ]e^{-1},1[$  is {\it translated} by $\psi$  if for some $t\in\R$
\be\psi({\rm Gr}_{\gamma})={\rm Gr}_{t+\gamma}\label{invcurve}
\ee
and {\it invariant} if $t=0$.  If ${\rm Gr}_{\g}$ satisfies (\ref{invcurve}),   there exists an orientation preserving  homeomorphism of the circle $g:\R/\Z\to\R/\Z$ such that $\psi(x,\gamma(x))=\psi(g(x),t+\gamma(g(x))$. If $t=0$ (resp. $t\ne 0$), we define  (resp. with  a clear abuse of language) the rotation number of ($\psi$ on) the invariant (resp. translated)  graph ${\rm Gr}_{\g}$ as the rotation number of the circle diffeomorphism $g$.  We say that $\psi$ has the {\it intersection property} if for any continuous  $\g:\R/\Z\to \R/\Z\times ]e^{-1},1[$, the curve ${\rm Gr}_{\g}:=\{(x,\g(x)):x\in\R/\Z\}$ intersects its image $\psi({\rm Gr}_{\g})$. Note the following important fact:  If $\psi$ has the intersection property, any translated graph by $\psi$ is {\it  invariant. }

\medskip
We state the  Translated Curve  Theorem by R\"ussmann \cite{Russmann}(which implies the Invariant Curve Theorem by Moser \cite{Moser62}):
\begin{theo}[R\"ussmann, \cite{Russmann}] \label{theo:moser} There exists  $k_{0}\in\N$ for which the following holds. Let $k\geq k_{0}$, $C,\mu>0$ and  $l:\R/\Z\to\R$ a $C^k$ map satisfying the {\it  twist condition} 
\be \min_{y} |\pa_{y}l(y)|>\mu>0   \qquad  \textrm{and}\qquad   \|l\|_{C^{k_{0}}}\leq C\label{6.33}\ee and define $$\psi_{0}:(x,y)\mapsto (x+l(y),y).$$
There exists $\e_{0}=\e_{0}(C,\mu)>0$ such that,  for any $C^k$ diffeomorphism $$\psi:\R/\Z\times ]e^{-1},1[\to\R/\Z\times\R$$   satisfying 
\be \|\psi-\psi_{0}\|_{C^{k_{0}}}<\e_{0},\label{*6.54}\ee  the diffeomorphism $\psi$ admits a set of positive Lebesgue measure of $C^{k-k_{0}}$ translated graphs contained in $(\R/\Z)\times ]e^{-3/4},e^{-1/4}[$. Moreover, all these translated graphs have Diophantine rotation numbers\footnote{They are in a fixed Diophantine class $DC(\kappa,\tau)$  (the exponent is $\tau$ and the constant $\kappa$)  that can be prescribed in advance once $\mu$ is fixed ($k_{0}$ then depends on $\tau$ and   $\e_{0}$ on $\kappa$ and $\tau$).\label{foot11}}.
\end{theo}

\subsection{The rescaled diffeomorphism $\mathring{f}_{\e,n}$}\label{sec:6.2}
Let $\bar f_{\e}$ be the renormalized map defined in Section {\ref{sec:5.4}} and  define $u_{\e},v_{\e}$ by
$$\bar f_{\e}(x,y)=(x+u_{\e}(x,y),y+v_{\e}(x,y)).$$
Since $\bar f_{\e}=\bar\eta_{\e}\circ T_{l_{\e}}$ ({\it cf.} (\ref{eq:prop:5.4})) one has using (\ref{labelab})
\begin{align*}u_{\e}(x,y)&=l_{\e}(y)+a_{\e}(x+l_{\e}(y),y)\\
v_{\e}(x,y)&=yb_{\e}(x+l_{\e}(y),y).
\end{align*}

\begin{figure}
\begin{center}

\begin{pspicture}(-1,-1)(10,5)
\psset{xunit=1.7cm,yunit=1.2cm}

\psline(1,0)(1,3.5)
\psline(5,0)(5,3.5)
\psline(0,0)(6,0)
\psline(1,1.5)(5,1.5)
\psline(1,2.5)(5,2.5)
\psline[linestyle=dashed](1,1.75)(5,1.75)
\psline[linestyle=dashed](1,2.25)(5,2.25)

\def\h{ x 1  sub 180 mul  sin 0.2 mul
}
\def\ha{\h 1.5 add}
\def\hb{\h 2.5 add}

\psplot{1}{5}{\ha}
\psplot{1}{5}{\hb}

\pscustom[fillstyle=hlines,hatchwidth=.1pt,linestyle=none]{%
\psplot{1}{5}{\hb}
\psline(5,2.5)(5,1.5)
\psplot{5}{1}{\ha}
\psline(1,1.5)(1,2.5)}

\pscustom[linestyle=none]{%
\psline(1,2.5)(5,2.5)
\psline(5,2.5)(5,1.5)
\psline(5,1.5)(1,1.5)
\psline(1,1.5)(1,2.5)}

\rput(5.7,-0.5){$\R/\Z$}
\rput(.8,3.5){$\R$}
\rput(.5,1.5){$e^{-(n+1)}$}
\rput(.5,2.5){$e^{-n}$}
\rput(1,-0.5){$0$}
\rput(5,-0.5){$1$}
\rput(3,0.5){$\bar f_{\e}(\R/\Z\times [e^{-n+1},e^{-n}]) $}
\psline[linewidth=0.1pt](3,1)(3.5,2)
\end{pspicture}
\end{center}
\caption{The diffeomorphism $\bar f_{\e}$ on $\R/\Z\times [e^{-(n+1)},e^{-n}]$. }\label{fig:4}
\end{figure}
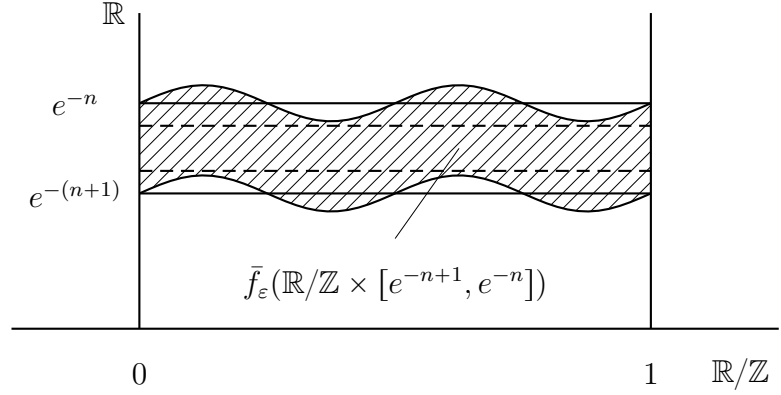

Now, let  $n\in\N^*$ large enough so that 
\be ]e^{-(n+1)},e^{-n}[\subset ]0,\d_{}[\label{condnepsilon}\ee
(the $\d_{}$ of (\ref{ee5.28}))
and introduce the rescaled $C^{k}$ diffeomorphism $\mathring{f}_{\e,n}$  defined on the annulus $\R/\Z\times ]e^{-1},1[$  by 
\be \mathring{f}_{\e,n}\mathop{=}_{defin.}\Lambda_{e^n}\circ \bar f_{\e}\circ \Lambda_{e^n}^{-1}.\label{6.34}
\ee
where $\Lambda_{e^n}:(x,y)\mapsto (x,e^n y)$. 
 Let us denote $$\mathring{f}_{\e,n}(x,y)=(x+u_{\e,n}(x,y), y+v_{\e,n}(x,y)).$$

A computation shows that:
\be \begin{cases}u_{\e,n}(x,y)&=l_{\e,n}(y)+a_{\e}(x+l_{\e,n}(y),e^{-n}y)\\
v_{\e,n}(x,y)&=yb_{\e}(x+l_{\e,n}(y),e^{-n}y)
\end{cases}\label{6.42}\ee 
where 
\be l_{\e,n}(y)=l_{\e}(e^{-n}y).\label{6.57**}
\ee

\medskip

We can now state the following important proposition the proof of which occupies the next subsection:
\begin{prop}\label{cor:6.2}Assume that $k\geq k_{0}+2$ ($k$ is the regularity in conditions \ref{iii0}-\ref{iii3} and $k_{0}$ is the one of Theorem \ref{theo:moser}). There exists $\e_{1}>0$ such that the following holds. If $|\e|\leq \e_{1}$ and $n\gg 1$, 
$\mathring{f}_{\e,n}$ admits  a set of positive Lebesgue measure of invariant $C^{k-k_{0}-2}$-graphs  in $\R/\Z\times ]e^{-1},1[$.
\end{prop}

\smallskip

\subsection{Proof of Proposition \ref{cor:6.2} }\label{s:6.3}
 \subsubsection{Twist condition for $l_{\e,n}$}\label{sec:6.3.1}
\begin{lem}\label{lemma:6.1}There exist  $C,\mu>0$ such that for any $\e$ small enough and any $n\gg1$  the map $l_{\e,n}$ satisfies the twist condition (\ref{6.33}) provided  $k\geq k_{0}+1$.
\end{lem}
\begin{proof}Using (\ref{5.27}), (\ref{6.57**}) we have
\begin{align*}l_{\e,n}(y)&=l_{\e}(e^{-n}y)\\
&=\frac{\sigma_{0,N_{}}(e^{-n}y)}{q'_{\e}(e^{-n}y)}+\frac{n}{q'_{\e}(e^{-n}y)}-\frac{\ln y}{q'_{\e}(e^{-n}y)} \quad \mod\ \Z\\
&=\frac{\s_{0,N_{}}(0)+n}{\l}-\frac{\ln y}{\l}+\th_{\e,n}(y)\quad\mod\ \Z
\end{align*}
where
$$ \|\th_{\e,n}\|_{C^{k-1}([e^{-1},1])}=O(e^{-n});$$
this last inequality is a consequence of the fact that $q_{\e}(s)=\l s+O(s^2)$ is continuous w.r.t. $\e$  ({\it cf.} condition \ref{iii1}) and of the fact that $\s_{0,N_{}}$ is $C^k$ ({\it cf.} Lemma \ref{lemma:4.3}).
In particular for some $C_{k}>0$ (depending on $\l$)
$$\| l_{\e,n} \|_{C^{k-1}}\leq C_{k}$$ and    since $\pa_{y} l_{\e,n}(y)=-1/ (\l y)+\pa_{y}\theta_{\e,n}(y)$ and $y\in ]e^{-1},1[$
$$|\pa_{y} l_{\e,n}(y)|\geq 1/(2\l) .$$
Hence (\ref{6.33})   holds uniformly in $\e,n$  with $C=C_{k_{0}+1}$ and $\mu=1/(2\l)$ as soon as $n$ is large enough.
\end{proof}

\subsubsection{ $\mathring{f}_{\e,n}$ is close to a twist}\label{sec:6.3.2}
We observe that from  (\ref{5.28})-(\ref{labelab}), (\ref{6.42}) and Lemma \ref{lemma:6.1}
one has uniformly in $n$
\be \lim_{\e\to0}\max(\|u_{\e,n}-l_{\e,n}\|_{C^{k-2}},\|v_{\e,n}\|_{C^{k-2}})=0.\label{6.43}\ee
In particular, if $n$ is large enough, inequality (\ref{*6.54}) is satisfied if $k\geq k_{0}+2$ with $\psi=\mathring{f}_{\e,n}$ and $\psi_{0}:(x,y)\mapsto (x+l_{\e,n}(y),y)$. 

\hfill $\Box$

\medskip  We see from \ref{sec:6.3.1}, \ref{sec:6.3.2} that,  if  
$$|\e|\leq \e_{1}\mathop{=}_{defin.} \e_{0}(C_{k_{0}+1},1/(2\l))$$
 and $n\gg1$,  the assumptions of Theorem \ref{theo:moser} are then satisfied by $\mathring{f}_{\e,n}$ with $k-2$ in place of $k$. Under these conditions, there thus exists a set $\mathring{\cG}_{\e,n}$ of $C^{k-k_{0}-2}$ $\mathring{f}_{\e,n}$-translated  graphs the union of which covers a set of positive Lebesgue measure in $(\R/\T)\times ]e^{-3/4},e^{-1/4}[$. 
We just have to check that these translated graphs are indeed invariant.
\subsubsection{$\mathring{f}_{\e,n}$-translated graphs are invariant}
Let $\mathring{\g}\subset (\R/\T)\times ]e^{-3/4},e^{-1/4}[$ be a $\mathring{f}_{\e,n}$-translated graph: $\mathring{f}_{\e,n}(\mathring{\g})=\mathring{\g}+(0,t)$ for some $t\in\R$. We shall prove that $t=0$. We can without loss of generality assume that $t\geq 0$ (the case $t\leq 0$ is treated in a similar way).

Formula (\ref{6.43}) shows that if $n\gg 1$ one has $\mathring{f}_{\e,n}(\mathring{\g})\subset (\R/\T)\times ]e^{-1},1[$. From the conjugation relation (\ref{6.34}) we see  that ({\it cf.} (\ref{condnepsilon})) 
$$\bar \g:=\Lambda_{e^n}^{-1}(\mathring{\g})\subset (\R/\Z)\times ]e^{-n-3/4}, e^{-n-1/4}[\subset (\R/\Z)\times ]0,\d[$$ 
is a $\bar f_{\e}$-translated graph such that 
$$\bar f_{\e}(\bar \g)=\bar\g+(0,e^{-n}t)\subset (\R/\Z)\times ]e^{-(n+1)},e^{-n}[\subset (\R/\Z)\times ]0,\d[.$$
Let $A$ be the open domain of $(\R/\Z)\times ]0,c[$ between $(\R/\Z)\times \{0\}$ and $\bar \g$. Since $t\geq 0$, one has $A\subset \bar f_{\e}(A)\subset (\R/\Z)\times ]0,c[$.

Assume by contradiction that $t>0$; then the set  $\bar f_{\e}(A)\setminus A$ contains a nonempty open set. 
We have seen  ({\it cf.} Proposition \ref{prop:main}) that $\bar f_{\e}$ preserves a probability measure $\bar \pi_{\e,y_{*}}$  with positive density defined on  $(\R/\Z)\times ]0,c[$, so  $\bar \pi_{\e,y_{*}}(\bar f_{\e}(A)\setminus A)>0$. But this contradicts the invariance of $\bar \pi_{\e,y_{*}}$ by $\bar f_{\e}$.

\hfill $\Box$

\medskip The proof of Proposition  \ref{cor:6.2}  is complete. \hfill $\Box$

\subsection{Invariant curves for $\bar f_{\e}$}
We can now state:
\begin{theo} \label{prop:6.3}Let $k\geq k_{0}+2$ and $|\e|\leq \e_{1}$.  There exists $\nu_{}\in ]0,\d_{}[$ such that for any $\nu\in]0,\nu_{}[$ there exists a set $\bar \cG_{\e,\nu}$ of $C^{k-k_{0}-2}$, $\bar f_{\e}$-invariant graphs contained in $(\R/\Z)\times ]e^{-1}\nu,\nu[$ such that 
$${\rm Leb}_{\R^2}\biggl( \bigcup_{\bar\g\in\bar\cG_{\e,\nu}} \bar\g\biggr)>0.$$
\end{theo}
\begin{proof}
We choose $n$ so that 
\be ]e^{-(n+1)},e^{-n}[\subset ]0,\nu[\label{n6.60}\ee
and we observe that when  $\nu\to 0$ one has  $n\to\infty$. 
Define
$$ \mathring{f}_{\e,n}=\Lambda_{e^n}\circ \bar f_{\e}\circ \Lambda_{e^n}^{-1}.$$
By Proposition \ref{cor:6.2}, there exists $\nu_{1}>0$ such that  if $\nu\in ]0,\nu_{1}[$ ($n$ satisfying (\ref{n6.60}) is then large enough), the diffeomorphism $\mathring{f}_{\e,n}$  admits  $C^{k-k_{0}-2}$-invariant curves in $\T\times ]e^{-1},1[$ covering a set of positive Lebesgue measure; hence $\bar f_{\e,k}$ has  $C^{k-k_{0}-2}$-invariant curves in $\T\times ]e^{-1}\nu,\nu[$ covering a set of positive Lebesgue measure. 
\end{proof}
We shall denote
$$\bar\cG_{\e}=\bigcup_{\nu\in]0,\nu_{1}[}\bar \cG_{\e,\nu}.$$

\begin{rem}\label{rem:6.1}
For all $\bar\g\in\bar\cG_{\e,\nu}$ the rotation number of the circle diffeomorphism  $\bar f_{\e}\ |_{\bar \g}$ is Diophantine in a fixed Diophantine class $DC(\kappa,\tau)$ (see Footnote \ref{foot11} in the statement of  Theorem \ref{theo:moser}). 
\end{rem}

\section{Invariant curves for $f_{\e}$}\label{sec:7}
We define 
$$r=k-k_{0}-2$$
 and assume that  $|\e|\leq \e_{1}$.

\medskip
Let $\bar \g\subset (\R/\Z)\times ]0,\d[$, $\bar\g\in\bar\cG_{\e}$, be a $C^r$ invariant graph for $\bar f_{\e}:(\R/\Z)\times ]0,\d[\to (\R/\Z)\times ]0,c[$. Note that there exists $\d_{1}>0$ such that $\bar \g\subset (\R/\Z)\times ]\d_{1},\d[$.

We can view $\bar \g$ as an invariant graph 
sitting in $([0,1[\times ]0,c[)/T_{1}$ (recall $T_{1}(x,y)=(x+1,y)$).
In particular one can find a  $C^r$, 1-periodic function  
$$\bar z:\R\to ([0,1[\times ]0,\d[)/T_{1}$$ such that for all $t$,  $\frac{d}{dt}\bar z(t)\ne 0$ and 
$$\bar \gamma= \bar z([0,1[),\qquad \bar z(0)\in \{0\}\times ]0,c[\qquad \lim_{t\to 1-}\bar z(t)=T_{1}(\bar z(0))\in \{1\}\times ]0,c[.$$

Let 
$$\hat \gamma=h_{\e}^{-1}(\bar\gamma)$$ 
where $h_{\e}$ was defined in Lemma \ref{lemma:5.2}.
 Since $\bar f_{\e}=h_{\e}\circ \hat f_{\e}\circ h_{\e}^{-1}$ ({\it cf. }(\ref{ee5.28}))  we see that 
 $$\hat \g\subset h_{\e}^{-1}((\R/\Z)\times ]\d_{1},\d[)\subset \ti\cF_{\e,c_{*}y_{*}}$$ is a $C^r$ compact, connected, 1-dimensional submanifold (without boundary) of  $\ti\cF_{\e,c_{*}y_{*}}/f_{\e}$   which  is invariant by $\hat f_{\e}:\ti\cF_{\e,c_{*}y_{*}}/f_{\e}\to\ti \cF_{\e,y_{*}}/f_{\e}$ .
Moreover, the function   $$\hat z\mathop{=}_{defin.} h_{\e}^{-1}\circ\bar z:\R\to \ti \cF_{\e,c_{*}y_{*}}/f_{\e}$$ is a $C^r$, 1-periodic function and
$$\hat \gamma= \hat z([0,1[),\qquad \hat z(0)\in L_{y_{*}}\qquad \lim_{t\to 1-}\hat z(t)=f_{\e}(\hat z(0))\in f_{\e}(L_{y_{*}}). $$

The main result of this section is the following proposition.
\begin{prop}\label{prop:7.2}The set
$$\hat \G=\bigcup_{n\in\Z}f_{\e}^n(\hat \g)\subset\R^2
$$
is an invariant $C^r$ curve for $f_{\e}$: it 
 is a compact, connected, 1-dimensional $C^r$ submanifold of $\R^2$ which is invariant by $f_{\e}$.
\end{prop}
 We give the proof of this proposition in subsection \ref{sec:7.2}.
\subsection{Preliminary results} 
We define the function $\hat Z:\R\to \R^2$
$$\forall\ t\in\R,\quad \hat Z(t)=f_{\e,k}^{[t]}(\hat z(t-[t]))$$
($[t]$ denotes the integer part of $t$ {\it i.e.} the unique integer such that $[t]\leq t<[t]+1$).
\begin{lemma}\label{lemma:Cr}The function $\hat Z:\R\to\R^2$ is $C^r$.
\end{lemma}
\begin{proof}Note that for $t\in [0,1[$, $\hat Z(t)=\hat z(t)$. Also, the very definition of $\ti \cF_{\e,y_{*}}/f_{\e,k}$ shows  that the function $\hat Z$ is $C^r$ on a neighborhood of  $t=1$. It is hence $C^r$ on $[0,2[$ and since for $j\in\Z$, $\hat Z(t+j)=f_{\e}^j(\hat Z(t))$, it is $C^r$ on $\R$.
\end{proof}

Let us set
$$\tau\mathop{=}_{defin.}\inf\{t\geq 1,\ \hat Z(t)\in \ti \cF_{\e,y_{*}}\}.$$
Note that
\be 2\leq \tau<\infty.\label{ee7.36}\ee
Indeed, the left hand side inequality is a consequence of the fact that $\ti \cF_{\e,y_{*}}\cap f_{\e}(\ti \cF_{\e,y_{*}})=\emptyset$. For the right hand side, we observe that  since $\hat Z(0)=\hat z(0)\in \ti \cF_{\e,c_{*}y_{*}}$, one has (see (\ref{deffhat})), $\hat Z(n_{\e}(\hat z(0)))=f_{\e}^{n_{\e}(\hat z(0))}(\hat z(0))\in \ti \cF_{\e,y_{*}}$, hence  $\tau\leq n_{\e}(\hat z(0))<\infty$.

\begin{lemma}\label{lemma:injective} The map $\hat Z:[0,\tau[\to\R^2$ is injective.
\end{lemma}
\begin{proof}
Assume by contradiction that  $\hat Z:[0,\tau[\to\R^2$ is not injective; then,  there exists $m_{i}\in\N$, $0\leq s_{i}<1$,   
 \be 0\leq m_{i}+s_{i}<\tau,\quad i=1,2,\qquad \hat Z(s_{1}+m_{1})=\hat Z (s_{2}+m_{2}).\label{eq:7.37}\ee  Hence $f_{\e}^{m_{1}}(\hat \gamma)\cap f_{\e}^{m_{2}}(\hat \gamma)\ne\emptyset$ and if $m:=m_{2}-m_{1}\geq 0$, $f_{\e}^{m}(\hat \g)\cap \hat \g\ne\emptyset.$ In particular, there exists $t\in [m,m+1[$ such that $\hat Z(t)\in\ti\cF_{\e,y_{*}}$  and then $t\geq \tau$. As a consequence $m> \tau-1$ and since $0\leq m<\tau$ ($m_{1},m_{2}$ are both in the interval $[0,\tau[$) one has $m=[\tau]$ hence $m_{2}=m=[\tau]$ and $m_{1}=0$.
We then have from (\ref{eq:7.37}), $\hat Z(s_{2}+[\tau])=\hat Z(s_{1})\in\ti\cF_{\e,y_{*}}$ (since $s_{1}\in [0,1[$) hence by the definition of $\tau$, $s_{2}+[\tau]\geq \tau$ which contradicts $m_{2}+s_{2}<\tau$.
\end{proof}
\begin{lemma}\label{item:7-2}  If for  some $t\geq 1$, $\hat Z(t)\in\ti\cF_{\e,y_{*}}$, then $\hat Z(t)\in\hat\gamma$.
\end{lemma}
\begin{proof} Indeed, writing $t=s+n$, $s\in [0,1[$, $n\in\N^*$, one has $\hat Z(t)=f_{\e}^n(\hat z(s))$. The integer $n\geq 1$ is thus a $m^{th}$ return time of $\hat z(s)$ in $\ti\cF_{\e,y_{*}}$, $\hat Z(t)=\hat f_{\e}^{m}(\hat z(s))$, and since $\hat \gamma$ is invariant by $\hat f_{\e}$,  it is readily seen by induction on $m$ that $f_{\e}^n(\hat z(s))\in \hat \gamma$.
\end{proof}
\begin{lemma}\label{lemma:e7.2} One has $\hat Z(\tau)=\hat z(0)$.
\end{lemma}
\begin{proof}From the definition of $\tau$ and  Lemma \ref{item:7-2} we have $\hat Z(\tau)\in{\rm closure}(\hat\gamma)\cap {\rm closure}(L_{y_{*}} \cup f_{\e}(L_{y_{*}})) $ hence $\hat Z(\tau)\in\{\hat z(0),f_{\e}(\hat z(0))\}$. To conclude we observe that one cannot have $\hat Z(\tau)=f_{\e}(\hat z(0))$ since otherwise one would have $\hat Z(\tau-1)=\hat z(0)\in\ti\cF_{\e,y_{*}}$ which contradicts the definition of $\tau$ (from  (\ref{ee7.36})  $\tau-1\geq 1$). 
\end{proof}
\begin{lemma}\label{lemma:7.2}The derivative of $\hat Z$ at $\tau$ is transverse to $L_{y_{*}}$.
\end{lemma}
\begin{proof}
1) If there exists a sequence $t_{n}\in\R$, $\lim t_{n}=\tau$ such that $Z(t_{n})\in\ti\cF_{\e,y_{*}}$, then from Lemma (\ref{item:7-2}) one has $Z(t_{n})\in\hat\gamma$ and consequently $(d\hat Z/dt)(\tau)$ is tangent to $\hat \gamma$, thus transverse to $L_{y_{*}}$.

2) Otherwise,  there exists an open interval $I\subset \R$, $I\ni\tau$,  such that   for all $t\in I\setminus\{\tau\}$,  $\hat Z(t)\notin\ti\cF_{\e,y_{*}}$
and $f_{\e}(\hat Z(t))\in \cF_{\e,y_{*}}$. From Lemma (\ref{item:7-2}) one then has for all $t\in I\setminus\{\tau\}$, $\hat Z(t+1)=f_{\e}(\hat Z(t))\in\hat\gamma$ (see item (\ref{i2}) of Subsection \ref{sec:glueing}) hence $Df_{\e}(f_{\e}(\hat Z(\tau)))\cdot (d\hat Z/dt)(\tau)$ is tangent to $\hat \g$ and in particular transverse to $f_{\e}(L_{y_{*}})$.  This implies that $(d\hat Z/dt)(\tau)$  is transverse to $L_{y_{*}}$.
\end{proof}

\begin{lemma}\label{lemma:7.4}One has $\hat Z([\tau,\tau+1[)=\hat Z([0,1[)$.
\end{lemma}
\begin{proof}
We define $s_{*}=\sup\{s\geq 0:\forall  t\in [\tau, \tau+s[, \ \hat Z(t)\in \cF_{\e,y_{*}}\}$. From Lemmata \ref{item:7-2} and \ref{lemma:7.2} one has: (a) $s_{*}>0$;  (b) for any $t\in [\tau,\tau+s_{*}[$, $\hat Z(t)\in \hat\gamma$; and (c)
$\hat Z(\tau+s_{*})\in f_{\e}(L_{y_{*}})\cap {\rm closure}( \hat\gamma)=f_{\e}(\hat z(0))=\hat Z(1)$. In particular, $\hat Z(\tau+s_{*}-1)=f_{\e}^{-1}(\hat Z(1))=\hat Z(0)\in\ti\cF_{\e,y_{*}}$ and by the definition of $\tau$, this implies $s_{*}\geq 1$. Now we notice that one cannot have $s_{*}>1$ because otherwise 
  $\tau+1\in [\tau,\tau+s_{*}[$ and by  definition of $s_{*}$, $\hat Z(\tau+1)\in\cF_{\e,y_{*}}$; but $\hat Z(\tau+1)=f_{\e}(\hat Z(\tau))$ and since $\hat Z(\tau)=\hat z(0)$  (Lemma \ref{lemma:e7.2}) one has  $\hat Z(\tau+1)=f_{\e}(\hat z(0))\notin \cF_{\e,y_{*}}$. We have hence proven that 
$s_{*}=1$. This implies that $\hat Z([\tau,\tau+1[)=\hat Z([0,1[)$.

\end{proof}

\subsection{Proof of Proposition \ref{prop:7.2}}\label{sec:7.2}
We first observe that 
\be \hat\Gamma=\bigcup_{n\in\Z}f_{\e}^n(\hat \g)=\hat Z(\R)=\bigcup_{n\in\Z}f_{\e}^n(\hat Z([0,\tau+1[)).\label{n7.63}\ee
Next we note that
\begin{enumerate}
\item \label{item1}One has $\hat Z ([0,\tau+1[)=\hat Z([0,\tau])$.
\item \label{item2}The set $\hat Z([0,\tau+1[)$ is $f_{\e}$-invariant. 
\end{enumerate}
Item (\ref{item1}) is a consequence of \begin{align*}\hat Z ([0,\tau+1[)&=\hat Z ([0,\tau])\cup \hat Z ([\tau,\tau+1[)\\
&=\hat Z([0,\tau])\cup \hat Z([0,1[) \qquad(\textrm{Lemma}\ \ref{lemma:7.4})\\
&=\hat Z([0,\tau])\qquad (1\leq \tau).
\end{align*}
Item (\ref{item2}) follows from Item (\ref{item1}) and
\begin{align*}f_{\e}(\hat Z([0,\tau+1[))&=f_{\e}(\hat Z ([0,\tau]) )\\
&=\hat Z([1,\tau+1])\\
&=\hat Z([1,\tau])\cup \hat Z([\tau,\tau+1])\\
&=\hat Z([1,\tau])\cup \hat Z([0,1])\qquad(\textrm{Lemma}\ \ref{lemma:7.4})\\
&=\hat Z([0,\tau])\qquad (1\leq \tau)\\
&=\hat Z([0,\tau+1[).
\end{align*}
Item (\ref{item2}) and (\ref{n7.63}) yield
$$\hat \Gamma=\hat Z([0,\tau+1[).$$
This last identity  shows that $\hat \Gamma$ is a connected, compact ({\it cf.} Item (\ref{item1})) subset of $\R^2$ which is $f_{\e}$-invariant.

\medskip Let us prove that $\hat \Gamma$ is a 1-dimensional submanifold of $\R^2$.
Since $\hat Z(\tau)=\hat Z(0)$ (Lemma \ref{lemma:e7.2}) one has 
$$\hat Z([0,\tau+1[)=\hat Z(]0,\tau+1[)=\hat Z(]0,\tau[)\cup \hat Z(]\tau-1,\tau+1[).$$ 
From Lemmata \ref{lemma:Cr}, \ref{lemma:injective} the set $\hat Z(]0,\tau[)$ is a 1-dimensional submanifold of $\R^2$ as well as the set $\hat Z(]\tau-1,\tau+1[)$ (note that $\hat Z(]\tau-1,\tau+1[)=f_{\e}(\hat Z([\tau-2,\tau[)$). The intersection of these two sets is  $\hat Z(]\tau,\tau+1[)$ and from Lemma \ref{lemma:7.4} it is  equal to $\hat Z(]0,1[)$ which is a 1-dimensional submanifold of $\R^2$. As a consequence the union $\hat Z(]0,\tau[)\cup \hat Z(]\tau-1,\tau+1[)$ is 1-dimensional submanifold of $\R^2$.

\medskip
This concludes the proof of Proposition \ref{prop:7.2}
\hfill $\Box$

\section{Proof of Theorem \ref{theo:2.1} (hence of Theorem \ref{theo:A})  }\label{sec:8}
  As we have mentioned in Section \ref{sec:2.4} Theorem \ref{theo:A} follows  from  Theorem \ref{theo:2.1}; we describe the  proof  of this latter result  in this Section. 

\bigskip Let  $r=k-k_{0}-2$,  $|\e|\leq \e_{1}$ and $\nu\leq \nu_{1}$. Theorem \ref{prop:6.3} yields a set $\bar \cG_{\e,\nu}$ of $C^r$, $\bar f_{\e}$-invariant graphs contained in $(\R/\Z)\times ]e^{-1}\nu,\nu[$, the union of which covers a set of positive Lebesgue measure.

In the previous section ({\it cf.} Proposition \ref{prop:7.2}), for all $\nu\in ]0,\nu_{1}[$, we have associated to each $\bar f_{\e}$-invariant graph $\bar \g\in\bar \cG_{\e,\nu}$  an $f_{\e}$-invariant $C^r$-curve 
\be \hat \Gamma=\bigcup_{n\in\Z}f_{\e}^n(\hat\g),\qquad\textrm{where}\quad \hat \g=h_{\e}^{-1}(\bar \g).\label{8.61*}\ee
We denote by $\hat \cG_{\e,\nu}$ the set of all such curves $\hat\Gamma$.

 To prove Theorem \ref{theo:2.1} we just have to prove that for all $\nu\in]0,\nu_{1}[$

\be(\textrm{Positive\  measure})\qquad {\rm Leb}_{\R^2}\biggl(\bigcup_{\hat \Gamma\in\hat \cG_{\e,\nu}}\hat \Gamma \biggr)>0 \label{eq:8.38}\ee
 and
 \be(\textrm{Accumulation})\qquad\lim_{\nu\to 0}\sup_{\hat\Gamma\in\hat\cG_{\e,\nu}} {\rm dist}(\hat \Gamma,\Sigma_{\e} )=0.\label{eq:8.39}
 \ee
 
\subsection{Proof of (\ref{eq:8.38}) (Positive measure)}
This is a consequence of 
 the  inclusion  ({\it cf.} (\ref{8.61*})) 
$$h_{\e}^{-1}\biggl(\bigcup_{\bar \g\in\bar \cG_{\e,\nu}}\bar \g \biggr)\subset \bigcup_{\hat \Gamma\in\hat \cG_{\e,\nu}}\hat \Gamma $$
and of the fact that  ${\rm Leb}_{2}\biggl(\bigcup_{\bar \g\in\bar \cG_{\e,\nu}}\bar \g\biggr)>0$ (this is the content of Theorem \ref{prop:6.3}).

\subsection{Proof of (\ref{eq:8.39}) (Accumulation)}

Let $\bar\g\in\bar\cG_{\e,\nu}$, $\bar\g\subset (\R/\Z)\times]0,\nu[$.
From the definition (\ref{e5.26}) of the diffeomorphism $h_{\e}$ we see that for some positive  constant $C_{\l}$ depending on $\l$ ({\it cf.} condition \ref{iii2})
$$\hat\g=h_{\e}^{-1}(\bar\g)\subset\{(x,y)\in \ti \cF_{\e,y_{*}},\ xy\in ]0,C_{\l}\nu[ \}.$$

On the other hand 
$$\hat \Gamma=\bigcup_{n\in\Z}f_{\e}^n(\hat \g)=\biggl(\bigcup_{n\in\Z}f_{\e}^n(\hat \g)\cap V_{}\biggr)\cup \bigcup_{\substack{n\in\Z\\ f_{\e}^n(\hat \g)\not\subset V_{}}}f_{\e}^n(\hat \g).$$
From condition  \ref{iii2}  one has 
$$\bigcup_{n\in\Z}f_{\e}^n(\hat\g)\cap V_{}\subset V_{}\cap \{(x,y),\ xy\in ]0,C_{\l}\nu[\}$$
hence, using Remark \ref{rem:2.2}
\be {\rm dist}\biggl( \bigcup_{n\in\Z} f_{\e}^n(\hat \g)\cap V_{} ,\Sigma_{\e}\cap V_{} \biggr)=o_{\nu}(1) \ (\textrm{uniform \ in }\ \hat\g).\label{8.41}\ee
Now, recalling  the definition (\ref{defN}) of the integer $N_{}$ of Section \ref{sec:4.2}  one has 
$$ \bigcup_{\substack{n\in\Z\\ f_{\e}^n(\hat \g)\not\subset V_{}}}f_{\e}^n(\hat \g)\subset \bigcup_{n=1}^{N_{}} f^{-n}_{\e}(\hat\g)  $$
and using the fact that ${\rm dist}(\hat \g,\Sigma_{\e}\cap [(x_{*},0),f_{\e}(x_{*},0)[)=o_{\nu}(1)$ one can see that  ($N_{}$ is fixed)
\be  {\rm dist}\biggl( \bigcup_{\substack{n\in\Z \\ f_{\e}^n(\hat\g)\not\subset V_{}}} f_{\e}^n(\hat \g) ,\Sigma_{\e}\cap \bigcup_{j=1}^{N_{}} f_{\e}^{-1}([(x_{*},0),f_{\e}(x_{*},0)[)) \biggr)=o_{\nu}(1)\label{8.42}\ee
where the previous limit is uniform in $\hat\g$.

Equations (\ref{8.41}), (\ref{8.42}) give 
$${\rm dist}(\hat \Gamma,\Sigma_{\e})=o_{\nu}(1).$$

\subsection{KAM circles for $f_{\e}$}\label{rem:8.1}
Let $\hat \Gamma$ be a $C^r$ invariant curve for $f_{\e}$ of the form (\ref{8.61*}) and $g_{\hat \Gamma}$ the restriction of $f_{\e}$ to $\hat \Gamma$. The map $g_{\hat\Gamma}$ can be identified with a circle diffeomorphism.  Similarly, the restriction of $\bar f_{\e}$ to the invariant curve $\bar \g$ yields a circle diffeomorphism $g_{\bar \g}$. 

Let $\hat\a$ and $\bar\a$ are the rotation numbers of $g_{\hat\Gamma}$ and $g_{\bar\g}$. 
\begin{lemma}One has $\{1/\hat\a\}=\bar\a$
(here $\{\cdot\}$ denotes  the fractional part). 
\end{lemma}
\begin{proof}We refer to the renormalization procedure defined in Sections \ref{sec:4} and \ref{sec:5}. Let $\hat J$ be  the arc $\ti \cF_{\e,y_{*}}\cap \hat \Gamma$. The  restriction on $\hat J$  of $\hat f_{\e}$, the first return map of $f_{\e}$ in  $\ti \cF_{\e,y_{*}}$, defines a $C^r$ diffeomorphism of the abstract circle $\hat J/f_{\e}$.  Classical arguments show that the rotation number of this circle diffeomorphism is equal to $\{1/\hat\a\}$. On the other hand, after normalization of $f_{\e}$ by $h_{\e}$ ({\it cf.} formula (\ref{ee5.28})), $\hat \G$ is transported to  $\bar \g$ and the $C^r$ diffeomorphism $\hat f_{\e}:\hat J/f_{\e}\to \hat J/f_{\e}$ to the circle diffeomorphism $\bar f_{\e}: \bar J/T_{1}\to \bar J/T_{1}$ where $\bar J=h_{\e}(\hat J)\subset \bar \g$ is a fundamental domain of $\bar f\ {|\ \bar \g}$. The rotation numbers of  $\hat f_{\e}:\hat J/f_{\e}\to \hat J/f_{\e}$ and $\bar f_{\e}: \bar J/T_{1}\to \bar J/T_{1}$ are hence equal. But the rotation number of  $\bar f_{\e}: \bar J/T_{1}\to \bar J/T_{1}$  is (same  argument as before) equal to $\{1/\bar\a\}$.
\end{proof}
Since $\bar\a$  can be chosen in a fixed Diophantine class $DC(\kappa,\tau)$ (see Remark \ref{rem:6.1}), the rotation number  $\hat\a$ is Diophantine with the {\it same} exponent $\tau$. By Herman-Yoccoz Theorem  on linearization of $C^r$-circle diffeomorphisms (\cite{Herman-ihes}, \cite{Y-84}), this implies that if $r$ is large enough (depending on $\tau$ which is fixed), the diffeomorphism $g_{\hat \Gamma}$ is {\it linearizable}; in other words, $\hat \Gamma$ is a {\it KAM} curve. On the other hand one has {\it a priori} no control on the Diophantine constant of $\hat\a$.

\bigskip This concludes the proof of Theorem \ref{theo:2.1} whence of Theorem \ref{theo:A}.\hfill  $\Box$

\section{Proof of Theorem \ref{theo:B}}\label{sec:9}
We construct in Subsection \ref{sec:9.1} a symplectic diffeomorphism  $f_{pert}$ admitting a separatrix $\Sigma$ (see Figure \ref{fig:*}) and depending on a (``large'') parameter $M$.  We renormalize $f_{pert}$ like in Sections \ref{sec:4} and \ref{sec:5} to get  a diffeomorphism $\bar{f}_{pert}$ of an open annulus $\R/\Z\times ]0,c[$.  We  prove  in Proposition \ref{prop:9.3} of Subsection \ref{sec:9.2}  that this renormalized diffeomorphism $\bar f_{pert}$ sends some  graphs projecting on a fixed interval $J_{M}$ (see (\ref{JM}) on graphs which project on the whole circle and which are below the initial  graphs we have started from, see Figure \ref{fig:7}.  We then iterate this procedure in Subsection \ref{sec:9.3} to find   an  orbit of $\bar f_{pert}$ accumulating the boundary $\R/\Z\times\{0\}$ of the aforementioned annulus: this prevents the existence of $\bar{f}_{pert}$-invariant  curves close to this boundary  and therefore of $f_{pert}$-invariant curves close to the separatrix $\Sigma$. The diffeomorphism $f_{pert}$ is the searched for example of Theorem \ref{theo:B}.

\subsection{Construction of the example}\label{sec:9.1}
We start with a smooth autonomous symplectic vector field of the form  $X_{0}=J\nabla H_{0}$ where $H_{0}:\R^2\to\R$ satisfies on some neighborhood $V$ of $o=(0,0)$
$$H_{0}(x,y)=x y\qquad \textrm{on} \ V $$
and has the property that $\Sigma=H_{0}^{-1}(H_{0}(0,0))$ is compact and connected. The set $\Sigma$ is a separatrix of 
$$f\mathop{=}_{defin.}\phi^1_{J\nabla H_{0}}$$ associated to the hyperbolic fixed point $o$.

Fixing $x_{*}>0$ small enough we  can define like in Section \ref{sec:4}, for $y_{*}>0$ small enough, a fundamental domain $\ti \cF_{y_{*}}=\cF_{y_{*}}\cup L_{y_{*}}\subset V$ where $\cF_{y_{*}}$ is defined by $(a)-(d)$ (sec. \ref{sec:4.1}) with $\phi^1_{J\nabla H_{0}}$ in place of $f_{\e}$. We can even assume that $\phi_{J\nabla H_{0}}^{-j}(\ti\cF_{y_{*}})\subset V$, $j=1,2$. There exists $c_{*}>0$ such that the first return map
$$\hat f:\ti\cF_{c_{*}y_{*}}\to \cF_{y_{*}}$$
is well defined. We can renormalize $f=\phi^1_{J\nabla H_{0}}$ like in Section \ref{sec:5} by first normalizing $f$ ({\it cf.} Lemma  \ref{lemma:5.2}):
\be h\circ f\circ h^{-1}=T_{1}\label{ee9.50}
\ee
where 
\be h:\cF_{y_{*}}\to [0,1[\times]0,c[ \label{9.67*}\ee is symplectic \footnote{See (\ref{e5.26}), (\ref{e5.23}) and the fact that we choose $q(s)=s$.}
and then  setting ({\it cf.} (\ref{ee5.28}))
\be \bar f\mathop{=}_{defin.}h\circ \hat f\circ h^{-1}:\R/\Z\times ]0,\d[\to \R/\Z\times ]0,c[.\label{ee9.51}\ee
By (\ref{5.27}) of Proposition \ref{prop:main} we have
\be \bar f=T_{l},\qquad l(y)=\s(y)-\ln y\label{ee9.52}\ee
for some smooth function $\s$.

We can assume that $h(\ti \cF_{y_{*}})=[0,1[\times ]0,c[$ and that $T^{-j}([0,1[\times ]0,c[)\subset h^{-1}(V)$, $j=1,2$.

\bigskip We now construct a symplectic perturbation $f_{pert}:\R^2\to\R^2$ of $f$ which admits $\Sigma$ as a separatrix. We shall need first the following lemma.

\begin{lemma}\label{lemma:9.1} There exist $b\in (0,1)$ and a  nonempty compact interval $I\subset ]0,1[$  such that for any $M>0$,  there exists a smooth function $\ph_{M}:\R\to\R$ satisfying
\begin{enumerate}
\item\label{ii1} $\ph_{M}\ |_{I}\leq -bM$;
\item\label{ii2} $\frac{ b^{-1} M}{|I|}\geq -\ph_{M}' \ |_{ I}\geq \frac{M}{|I|}$;
\item\label{ii3} the map $s_{M}:\R\to\R$ defined by
$$s_{M}(t)=\int_{0}^te^{\ph_{M}(u)}du$$
is an increasing  smooth diffeomorphism of $\R$ that coincides with the identity on $\R\setminus[0,1]$.
\end{enumerate}
\end{lemma}
\begin{proof}See the Appendix \ref{sec:A3}.
\end{proof}

Let $\chi:\R\to\R$ be a smooth function equal to 1 on $[-c/2,c/2]$ and to 0 on $\R\setminus[-(3/4)c,(3/4)c]$ and 
 define
\be S_{M}(x,y)=(s_{M}(x)y)\chi (y)+xy(1-\chi(y)).\label{9.72}\ee
The canonical (hence symplectic) mapping $g_{M}$ associated to $S_{M}$:
\be g_{M}(x,y)=(\ti x, \ti y)\quad \iff \quad \begin{cases}&x=\frac{\pa S_{M}}{\pa y}(\ti x,y)\\
&\ti y=\frac{\pa S_{M}}{\pa \ti x}(\ti x,y)
\end{cases}\label{9.73}
\ee
is equal to the identity on $(\R\setminus[0,1])\times [-c,c]$ and
satisfies for $(x,y)\in [0,1[\times ]0,c/2[$
\be \begin{cases}&\ti x=s_{M}^{-1}(x)\\
&\ti y=s'_{M}\circ s_{M}^{-1}(x)y.
\end{cases}\label{eqg}
\ee
The following symplectic perturbation of $f$
\begin{align*}f_{pert}:&\mathop{=}_{defin.}h^{-1}\circ (g_{M}\circ T_{1})\circ h\\
&=(h^{-1}\circ g_{M}\circ h)\circ f
\end{align*}
(recall $h$ satisfies (\ref{ee9.50}))
 is thus defined on $\R^2$ and coincides with $f$ outside $f^{-1}(\ti \cF_{y_{*}})$. Moreover, since $g_{M}(\R\times \{0\})=\R\times\{0\}$
 $$\Sigma \ \ \textrm{is \ a \ separatrix\ for \ } f_{pert}.$$

 Now,
  since  $\ti\cF_{y_{*}}$ is a fundamental domain for $f_{pert}$ ($f_{pert}$ coincide with $f$ on $\ti \cF_{y_{*}}$), for some $c_{pert}>0$ small enough, the first return map
 $$ \hat f_{pert}:f_{pert}^{-1}(\ti \cF_{c_{pert}y_{*}})\to f_{pert}^{-1}(\ti \cF_{y_{*}}),$$
is well defined and satisfies
$$\hat f_{pert}=(\hat f\circ f^{-1})\circ f_{pert}.$$
In particular, on 
$$[-1,0[\times ]0,c/2[$$
 one has ({\it cf.} (\ref{ee9.51}), (\ref{ee9.52}))
\begin{align}\bar f_{pert}\mathop{=}_{defin.}h\circ \hat f_{pert}\circ h^{-1}&=\bar f\circ T_{-1}\circ g_{M}\circ T_{1}\label{e9.53ante}\\
&=T_{l-1}\circ g_{M}\circ T_{1}\label{e9.53}
\end{align}
and
$$\bar f_{pert}:\R\times ]0,c/2[\to \R\times ]0,c[\quad \textrm{satisfies}\quad \bar f_{pert}\circ T_{1}=T_{1}\circ \bar f_{pert};$$
in particular, it defines a smooth map $(\R/\Z)\times ]0,c/2[\to (\R/\Z)\times ]0,c[$.

Note that since $g_{M}$ is the identity outside $[0,1]\times [-c,c]$, it admits a $T_{1}$-periodization $\ti g_{M}:\R\times [-c,c]\to \R\times [-c,c]$ (which means that $\ti g_{M}$ and $g_{M}$ coincide on $[0,1]\times [-c,c]$ and $\ti g_{M}$ commutes with   $T_{1}$). This $\ti g_{M}$ is defined by the same formula  (\ref{9.73}) as $g_{M}$ where now the  new function $\ti s_{M}$ involved in (\ref{9.72}) is the $\Z$-periodization of $s_{M}$. To simplify the notation we shall continue to  denote $\ti g_{M}$ and $\ti s_{M}$ by $g_{M}$ and $s_{M}$.

Let 
\be t:=t(x):=s_{M}^{-1}(x+1).\label{a9.54}\ee

\begin{figure}
\begin{center}
\begin{pspicture}(-1,-1)(10,6)
\psset{xunit=1cm,yunit=1cm}

\psline(1,0)(7,0)
\psline(4,0)(4,3.7)

\psbezier(2,0)(1.9,2)(1.9,2)(1,3.5)
\psbezier(6,0)(6.1,2)(6.1,2)(7,3.5)
\pscurve[linestyle=dotted](1,3.5)(1.5,3.7)(2,3.6)(2.5,3.7)(3,3.6)(3.5,3.7)(4,3.7)
\pscurve[linestyle=dotted](4,3.7)(4.5,3.7)(5,3.6)(5.5,3.7)(6,3.6)(6.5,3.7)(7,3.5)

\pscurve[linestyle=dashed](8,4.5)(7.5,2.5)(7,1.5)(6,1) 
\pscurve[linestyle=dashed](2,1)(1,1.5)(0.5,3)(0.5,4.5)
\psdot(3.5,1)
\rput(3.5,0.6){$f(z)$}
\psdot(3.2,1.3)
\rput(3,1.7){$f_{pert}(z)$}
\psdot(5,1)
\rput(5.2,0.7){$z$}
\psdot(1,1.5)
\rput(1,1){$f^2(z)$}
\psdot(7,1.5)

\rput(5,3.2){$f^{-1}(\cF_{y_{*}})$}
\rput(3,3.2){$\cF_{y_{*}}$}

\psline[linestyle=dashed](2,1)(6,1)

\psline(1.9,1)(2.2,1)
\psline(4.1,1)(3.8,1)

\def\Pa{x  2.2 neg  add  }
\def\Paa{\Pa  2 exp }
\def\Pb{x 3.8 neg add }
\def\Pd{x 3 neg add}
\def\Pbb{\Pb  2 exp }
\def\Paabb{\Paa \Pbb mul}
\def\Paabbc{\Paabb \Pd mul}

\psplot{2.2}{3.8}{x  2.2 neg  add  2 exp    x 3.8 neg add 2 exp mul x 3 neg add mul   4 mul  1 add }

\end{pspicture}
\end{center}
\caption{The perturbed map $f_{pert}$.}\label{fig:*}
\end{figure}
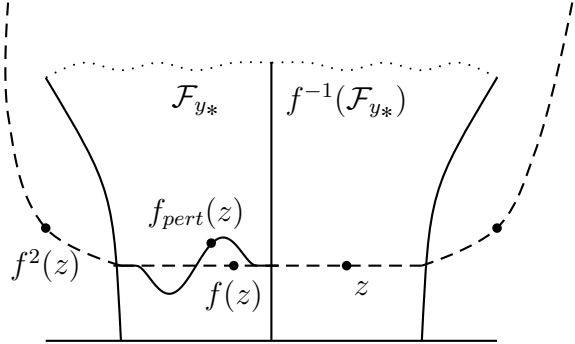

\begin{lemma}For $(x,y)\in [-1,0[\times ]0,c/2[$, the point $(\bar x,\bar y):=\bar f_{pert}(x,y)$ satisfies with the notation (\ref{a9.54})
\be \begin{cases}
&\bar x=t-1+\s\biggl(s_{M}'(t)\times y\biggr)-\ln(s_{M}'(t))-\ln y\\
&\ln \bar y=\ln(s_{M}'(t))+\ln y.
\end{cases}\label{9.53}\ee
\end{lemma}
\begin{proof}
Let  $(x,y)\in [-1,0[\times ]0,c/2[ $; with the notations $(x_{1},y_{1})=(g_{M}\circ T_{1})(x,y)=g_{M}(x+1,y)$, one has from (\ref{e9.53}) $(\bar x,\bar y)=T_{l-1}(x_{1},y_{1})$ and from (\ref{eqg}), (\ref{ee9.52})
$$\begin{cases}&x_{1}= s_{M}^{-1}(x+1)\\
&y_{1}=s_{M}'\circ s_{M}^{-1}(x+1)\times y
\end{cases}\quad\textrm{and}\quad \begin{cases}&\bar x=x_{1}-1+\s(y_{1})-\ln y_{1}\\
&\bar y=y_{1}.
\end{cases}$$
hence (\ref{9.53}).
\end{proof}
\subsection{Image of a piece of graph by $\bar f_{pert}$}\label{sec:9.2}
We take  $M>0$ (from Lemma \ref{lemma:9.1}) large enough and we define
\be J_{M}=s_{M}(I)-1\subset [-1,0[
\label{JM}
\ee
where $I$ is the interval introduced in Lemma \ref{lemma:9.1}.

\medskip
If $y:J\to\R_{+}^*$, $x\to y(x)$ is a differentiable function, we denote by    $\gamma_{J,y}$  its graph:
$$\gamma_{J,y}=\{(x,y(x)),\ x\in J\}  \subset [-1,0[\times ]0,c/2[.$$

\begin{prop}\label{prop:9.3}There exists a constant $y_{pert}>0$ for which the following holds. Assume that  $y:J_{M}\to]0,y_{pert}[$, $x\mapsto y(x)$ is a differentiable function such that 
$$\forall\ x\in J_{M},\qquad \biggl|\frac{d\ln y}{dx}+1\biggr|\leq 1/2.$$
Then, $\bar f_{pert}(\g_{J_{M},y})+(\Z,0)$ contains the graph $\g_{[-1,0[,\bar y}$ of a differentiable function  $\bar y:[-1,0[\to \R_{+}^*$ (see Figure \ref{fig:7})
$$\g_{[-1,0[,\bar y}=\{(\bar x,\bar y(\bar x)),\ \bar x\in [-1,0[\},$$ such that 
\begin{align}
&\forall\ \bar x\in [-1,0[,\qquad \biggl|\frac{d\ln \bar y}{d\bar x}+1\biggr|\leq 1/2;\label{9.51}\\
&\sup_{\bar x\in [-1,0[} \ln\bar y(\bar x) \leq \sup_{x\in J_{M}}\ln y(x)-bM.\label{9.52}
\end{align}
Moreover, for some interval $J_{M}^{1}\subset J_{M}$ one has
\be \gamma_{[-1,0[,\bar y}=\bar f_{pert}(\gamma_{J_{1},y}).\label{9.75*}\ee
\end{prop}
We prove this Proposition in subsection \ref{sec:proofprop9.3}. 
\subsubsection{Preliminary results}

If we introduce the variable
$$\ph=\ln(s_{M}'(t))=\ln s_{M}'\circ s_{M}^{-1}(x+1)\qquad (\textrm{recall}\ t=s_{M}^{-1}(x+1))$$
we can write (\ref{9.53}) as 
\be (\bar x,\bar y)=f_{pert}(x,y(x))\iff \begin{cases}
&\bar x=t-1+\s\biggl(e^\ph\times y(x)\biggr)-\ph-\ln y(x)\\
&\ln \bar y=\ph+\ln y(x).
\end{cases}\label{9.53bis}\ee
Note that the maps  $I\ni t\mapsto \ph=\ln s'_{M}(t)\in \ph(I)$ and $J_{M}\ni x\mapsto \ph= \ln s_{M}'\circ s_{M}^{-1}(x+1)\in \ph_{M}(I)$ are smooth diffeomorphisms. In particular, the maps $\ph_{M}(I)\ni\ph\mapsto \bar x$ and $\ph_{M}(I)\ni\ph\mapsto \ln y$,  $\ph_{M}(I)\ni\ph\mapsto \ln \bar  y$ are well defined and smooth.

\begin{lemma}\label{lemma:9.4}For any $\ph$ such that $t\in I$ one has
$$\biggl |\frac{dt}{d\ph}\biggr|\leq |I|/M\leq 1/4.$$
\end{lemma}
\begin{proof}This follows from the identity (recall $\ph=\ln(s_{M}'(t))$, $s_{M}'=e^{\ph_{M}}$)
$$\frac{dt}{d\ph}=\frac{1}{d\ph/dt}=\frac{1}{\ph_{M}'(t)}$$
and the estimates given by the second item of Lemma \ref{lemma:9.1} ($M$ is assumed to be large enough).
\end{proof}
\begin{lemma}\label{lemma:9.5}One has 
\begin{align}&\sup_{\ph_{M}(I)}\biggl|\frac{d\bar x}{d\ph}+1\biggr|\leq 1/4,\label{ee9.54}\\
&\sup_{\ph_{M}(I)}\biggl|\frac{d\ln \bar  y}{d\ph}-1\biggr|\leq 1/4.\label{ee9.55}
\end{align}
\end{lemma}
\begin{proof}Indeed, from (\ref{9.53bis})
\begin{align*}\frac{d\bar x}{d\ph}&=\frac{dt}{d\ph}+e^\ph\s'(e^\ph\times y)\frac{dy}{d\ph}-1-\frac{d\ln y}{d\ph}\\
&=\frac{dt}{d\ph}+ye^\ph\s'(e^\ph\times y)\frac{d\ln y}{d\ph}-1-\frac{d\ln y}{d\ph}\\
&=-1+A
\end{align*}
with 
$$A=\frac{dt}{d\ph}+ye^\ph\s'(e^\ph\times y)\frac{d\ln y}{d\ph}-\frac{d\ln y}{d\ph}$$
Note that (recall $x=s_{M}(t)-1$, $s_{M}'=e^{\ph_{M}}$)
$$\frac{d\ln y}{d\ph}=\frac{d\ln y}{dx}\frac{dx}{dt}\frac{dt}{d\ph}=\frac{d\ln y}{dx}e^\ph\frac{dt}{d\ph}$$
so, by Lemma \ref{lemma:9.4}
$$|A|\leq (|I|/M)+\biggl(ye^{-bM}\|\s'\|_{0}+1\biggr)e^{-bM}( |I|/M)\biggl|\frac{d\ln y}{dx}\biggr|
$$
and if $M$ is large enough 
\be |A|\leq 1/4.\label{estA}\ee

In a similar way 
$$\frac{d\ln \bar y}{d\ph}=1+\frac{d\ln y}{dx}\frac{dx}{dt}\frac{dt}{d\ph}=1+\frac{d\ln y}{dx}e^\ph\frac{dt}{d\ph}=1+B$$
with \be |B|\leq 2e^{-bM}\times (1/4)\leq 1/4\qquad (M\gg 1).\label{estB}\ee
\end{proof}
\subsubsection{Proof of Proposition \ref{prop:9.3}}\label{sec:proofprop9.3}
From (\ref{ee9.54}) of Lemma \ref{lemma:9.5} we see that the map $\ph_{M}(I)\ni \ph\mapsto \bar x\in\R$ is a diffeomorphism onto its image $\bar J_{M}\subset\R$,  hence the maps $J_{M}\ni x\mapsto \bar x\in \bar J_{M}$ and $I\ni t\mapsto \bar x\in\bar J_{M}$ are diffeomorphisms. Note that from (\ref{ee9.54}) one has
$$|\bar J_{M}|\geq (3/4)|\ph_{M}(I)|$$
and from item (\ref{ii2}) of Lemma \ref{lemma:9.1} one has
\be |\bar J_{M}|\geq (3/4)(M/|I|)\times |I|>2;\ee
 there thus exists an interval $J_{M}^{1}\subset J_{M}$ such that 
the map $J_{M}^{1}\ni x\mapsto \bar x\in n+ [-1,0[$ (for some $n\in\Z$) is a differentiable  homeomorphism. Replacing $\bar y(\bar x)$ by $\bar y (\bar x+n)$  shows (\ref{9.75*}).

We now prove  (\ref{9.51}): for  $\bar x\in [-1,0[$
\be\biggl| \frac{d\ln \bar y}{d\bar x}+1\biggr|\leq 1/2.
\ee
Indeed, let $I_{1}\subset I$ be the image of $[0,1[$ by $\bar J_{M}\ni \bar x\mapsto t\in I$;  from Lemma \ref{lemma:9.5}, for any $\ph\in \ph_{M}(I_{1})$ one has for some $A,B\in [0,1/4]$
$$\frac{d\bar x}{d\ph}=-1+A,\qquad \frac{d\ln \bar y}{d\ph}=1+B$$
so that
$$\biggl|\frac{d\ln \bar y}{d\bar x}+1\biggr|=\biggl| \biggl(\frac{d\ln\bar y}{d\ph}/\frac{d\bar x}{d\ph}\biggr)+1\biggr|=\biggl|\frac{1+B}{-1+A}+1\biggr|\leq 1/2.$$
The preceding discussion  shows that the map $\bar y:[-1,0[\ni \bar x\mapsto \bar y(\bar x)$ is a well defined differentiable function, that its graph is included in $f_{pert}(\gamma_{J_{M},y})+(\Z,0)$ and that (\ref{9.51}) holds.

There remains to prove (\ref{9.52}). By the second equality of (\ref{9.53bis}), if $(\bar x,\bar y(\bar x))=f_{pert}(x,y)$ one has
$$\ln\bar y(\bar x)\leq \ln y(x)-bM\leq \sup_{x\in J_{M}}\ln y-bM$$
and as a consequence since the map  $J_{M}\supset J_{M}^{1}\ni x\mapsto \bar x\in [-1,0[$ is a bijection, (\ref{9.52}) holds.\hfill $\Box$

\begin{figure}
\begin{center}
\begin{pspicture}(-1,-1)(10,5)
\psset{xunit=1.7cm,yunit=1.2cm}

\psline(1,0)(1,3.5)
\psline(5,0)(5,3.5)
\psline(0,0)(6,0)

\psline[linestyle=dotted](2,0)(2,3.5)
\psline[linestyle=dotted](2.5,0)(2.5,3.5)
\psline[linewidth=2pt](2,0)(2.5,0)
\psline[linestyle=dotted](1,1.75)(5,1.75)
\psline[linestyle=dotted](1,2.25)(5,2.25)

\psline[linestyle=dotted](1,2.25)(5,1.75)
\psline(2,2.13)(2.5,2.07)

\rput(2.2,-.4){$J_{M}$}

\psline[linestyle=dashed](2.2,1.5)(5,1.2)
\psline[linestyle=dashed](1,1.2)(5,0.8)
\psline[linestyle=dashed](1,0.8)(2.35,0.69)

\def\h{ x 1  sub 180 mul  sin 1 mul
}
\def\ha{\h 1.75 add}
\def\hb{\h 2.25 add}
\rput(1.6,1.95){$\g_{J_{M},y}$}

\rput(4,0.4){$\bar f_{pert}(\g_{J_{M},y})$}

\rput(5.7,-0.5){$\R/\Z$}
\rput(.8,3.5){$\R$}
\rput(.5,1.75){$e^{-(n+1)}$}
\rput(.5,2.25){$e^{-n}$}
\rput(1,-0.5){$-1$}
\rput(5,-0.5){$0$}
\end{pspicture}
\end{center}
\caption{The image of the graph $\g_{J_{M},y}$ by the  diffeomorphism $\bar f_{pert}$.}\label{fig:7}
\end{figure}
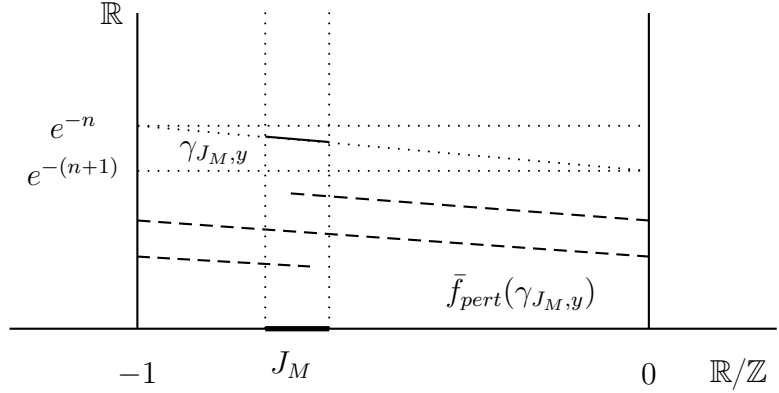

\subsection{End of the proof of Theorem \ref{theo:B}}\label{sec:9.3}

\medskip We shall prove that if $M$ is large enough, the diffeomorphism $f_{pert}$ constructed in Subsection \ref{sec:9.1} provides the searched for example of Theorem \ref{theo:B}.

\medskip
Let $M$ be large enough and $y_{0}\in ]0,y_{pert}[$;  we define the function
$$y_{0}:[-1,0[\to \R,\quad x\mapsto y_{0}e^{-x}.$$
Using inductively Proposition \ref{prop:9.3} we construct differentiable functions 
$$y_{n}:[-1,0[\to \R$$
such that  for every $n\in\N^*$
\begin{align}&\forall\ x\in J_{M},\qquad \biggl|\frac{d\ln y_{n}}{dx}+1\biggr|\leq 1/2\\
&\gamma_{[-1,0[,y_{n}}\subset \bar f_{pert}(\gamma_{J_{M},y_{n-1}})+(\Z,0)\label{9.84*}\\
&\sup_{x\in [-1,0[} \ln y_{n}(x) \leq \sup_{x\in J_{M}}\ln y_{n-1}(x)-bM.\label{9.85*}
\end{align}
Inclusion (\ref{9.84*}) implies the existence of a decreasing sequence of nonempty compact intervals $K_{n}\subset J_{M}$ such that 
$$\gamma_{[-3/4,-1/4],y_{n}}=\bar f_{pert}^n(\gamma_{K_{n},y_{0}})\mod (\Z,0).$$
In particular if $x_{\infty}\subset \bigcap_{n\in\N^*}K_{n}$ one has 
\be \forall \ n\in\N^*,\quad \bar f_{pert}^n((x_{\infty},y_{0}))\in \gamma_{[-3/4,-1/4],y_{n}}\subset \gamma_{[-1,0[,y_{n}}  \mod (\Z,0)\label{9.86*}.\ee
From (\ref{9.85*})
$$\sup_{x\in [-1,0[} y_{n}(x)\leq e^{-nbM}y_{0}$$
hence, using (\ref{9.86*}) we see that $ \bar f_{pert}^n((x_{\infty},y_{0}))$ accumulates $\R\times \{0\}$:
\be  \bar f_{pert}^n((x_{\infty},y_{0}))\in [-1,0[\times ]0,e^{-nbM}y_{0}[ \mod (\Z,0).\ee

As a consequence of (\ref{e9.53ante})  and of the fact that for some constant $C>0$ 
 $$\forall\ \nu \in ]0,c[,\qquad h^{-1}([-1,0[\times ]0,\nu[)\subset f_{pert}^{-1} (\hat \cF_{C\nu}),$$ 
 (this is due  to the fact that the diffeomorphism  $h$ given by (\ref{9.67*}) is indeed defined on a neighborhood of $\ti \cF_{y_{*}}$) one has 
$$\hat f_{pert}^n(h^{-1}(x_{\infty},y_{0}))\in f^{-1}_{pert}( \hat\cF_{Ce^{-nbM}y_{0}}).$$
Since $\hat f_{pert}$ is the first return map of $f_{pert}$ in $f^{-1}_{pert}(\hat \cF_{y_{*}})$, there exists a sequence
 $(p_{n})_{n\in\N}\in \N^\N$, $\lim_{n\to\infty} p_{n}=\infty$ such that 
\be f_{pert}^{p_{n}}(h^{-1}(x_{\infty},y_{0}))\in f^{-1}_{pert}( \hat\cF_{Ce^{-nbM}y_{0}}).\label{9.90*}\ee

But this last fact prevents the existence of invariant circles in $\Delta_{\Sigma}$ accumulating the separatrix $\Sigma$ of $f_{pert}$. More precisely, let $W$ be a neighborhood of $\Sigma$ in $\Sigma\cup\Delta_{\Sigma}$ (we recall that $\D_{\Sigma}$ is the bounded connected component of $\R^2\setminus\Sigma$) such that
$$ h^{-1}(x_{\infty},y_{0})\notin W.
$$
We claim that $W\setminus\Sigma$ does not contain any $f_{pert}$-invariant circle $\Gamma$.
Indeed, if this were not the case, the topological annulus $\cA\subset W$ having $\Sigma$ and $\Gamma$ for boundaries would  be $f_{pert}$-invariant (by topological degree theory). But this  is impossible since one would have at the same time
$$ h^{-1}(x_{\infty},y_{0})\notin \cA\qquad \textrm{and}\qquad f_{pert}^{p_{n}}(h^{-1}(x_{\infty},y_{0}))\in \cA
$$
for some large $p_{n}$ (see (\ref{9.90*}).

\hfill $\Box$

\begin{figure}
\begin{center}
\begin{pspicture}(-1,-1)(10,5)
\psset{xunit=1.7cm,yunit=1.2cm}
\psline(1,0)(1,3.5)
\psline(5,0)(5,3.5)
\psline(0,0)(6,0)
\psline(1,1.75)(5,1.75)
\psline(1,2.25)(5,2.25)
\def\h{ x 1  sub 180 mul  sin 1 mul
}
\def\ha{\h 1.75 add}
\def\hb{\h 2.25 add}
\psplot{1}{5}{\ha}
\psplot{1}{5}{\hb}

\pscustom[fillstyle=hlines,hatchwidth=.1pt,linestyle=none]{%
\psplot{1}{5}{\hb}
\psline(5,2.5)(5,1.5)
\psplot{5}{1}{\ha}
\psline(1,1.5)(1,2.5)}

\pscustom[linestyle=none]{%
\psline(1,2.5)(5,2.5)
\psline(5,2.5)(5,1.5)
\psline(5,1.5)(1,1.5)
\psline(1,1.5)(1,2.5)}

\rput(5.7,-0.5){$\R/\Z$}
\rput(.8,3.5){$\R$}
\rput(.5,1.75){$e^{-(n+1)}$}
\rput(.5,2.25){$e^{-n}$}
\rput(1,-0.5){$0$}
\rput(5,-0.5){$1$}
\rput(3,0.5){$\bar {\bar f}_{pert}(\R/\Z\times [e^{-(n+1)},e^{-n}]) $}
\end{pspicture}
\end{center}
\caption{The diffeomorphism $\bar{\bar f}_{pert}$ on $\R/\Z\times [e^{-(n+1)},e^{-n}]$. Compare with Figures \ref{fig:7} and  \ref{fig:4}. }\label{fig:5}
\end{figure}
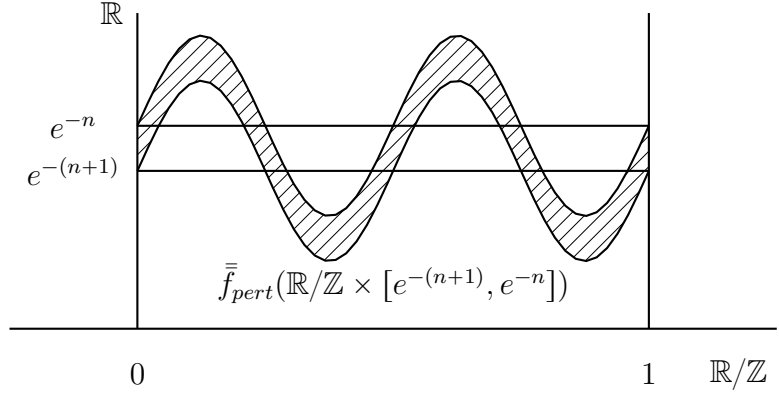
\begin{rem} If we define the renormalization $\bar {\bar f}_{pert}$ of $f_{pert}$ by considering the first return map of $f_{pert}$ in  $\cF_{y_{*}}$ instead of $f_{pert}^{-1}(\cF_{y_{*}})$ as we have done to construct $\bar f_{pert}$, the dynamics of $\bar{\bar f}_{pert}$ looks like more the one pictured in Figure \ref{fig:5}. The comparison of this picture and the one  of Figure  \ref{fig:4} illustrates the effect of the perturbative assumption in Theorem \ref{theo:A}.
\end{rem}

\appendix

\section{Proof of Lemma \ref{lemma:Birkhoff} }\label{appendix:Birkhoff+Sternberg}
We write for $j\geq 2$
$$H_{\e}^t(z)=\l_{\e} z_{1}z_{2}+\sum_{2\leq i\leq [j/2]}a_{\e,i}\times(z_{1}z_{2})^{i}+\sum_{\substack{i_{1},i_{2}\in\N\\ i_{1}+i_{2}=j+1 }}h_{\e,i_{1},i_{2}}(t)z_{1}^{i_{1}}z_{2}^{i_{2}}+O^{j+2}(z)$$
where $a_{\e,i}\in\R$ and the $h_{\e,i_{1},i_{2}}(\cdot)$ are smooth 1-periodic functions. 
We define
$$H_{\e,2}(z)=\l_{\e} z_{1}z_{2}.$$
We first observe that if $G_{\e}^t$ is a solution of
\be \begin{cases}&G_{\e}^t(z)=O^{j+1}(z)\\
&H_{\e}^t(z)+\pa_{t}G_{\e}^t(z)+\{G_{\e}^t,H_{\e,2}^t\}(z)=\ti q_{\e}(z_{1}z_{2})\end{cases}\label{A.48}\ee
for some  $\ti q_{\e}(u)=\l_{\e}u+\sum_{2\leq i\leq [(j+1)]/2}\ti a_{\e,i}\times u^{i}$, $\ti a_{\e,i}\in\R$,
then $G_{\e}^t$ solves (\ref{eq:2.8}). We then have to solve (\ref{A.48}) for some $\ti q_{\e}$ and some $G_{\e}^t$ of the form
\begin{align*}
&\ti q_{\e}(u)=\l_{\e}u+\sum_{2\leq i\leq [(j+1)]/2}\ti a_{\e,i}\times u^{i}\\
&G_{\e}^t(z)=\sum_{i_{1}+i_{2}=j+1}g_{\e,i_{1},i_{2}}(t)z_{1}^{i_{1}}z_{2}^{i_{2}}=O^{j+1}(z)\end{align*}
where the $g_{\e,i_{1},i_{2}}(\cdot)$ are 1-periodic.
This  amounts to finding 1-periodic solutions to the equations
\begin{align} & h_{\e,i_{1},i_{2}}(t)+\pa_{t}g_{\e,i_{1},i_{2}}(t)-\l_{\e}(i_{1}-i_{2})g_{\e,i_{1},i_{2}}(t)=0\quad & \textrm{if}\ i_{1}\ne i_{2},\label{A49}\\
 & h_{\e,i,i}(t)+\pa_{t}g_{\e,i,i}(t)=\ti a_{\e,i} \quad &\textrm{if}\ i_{1}= i_{2}=i,\label{A50}
\end{align}
for each couple $(i_{1},i_{2})\in\N^2$ such that $i_{1}+i_{2}=j+1$.
Note that in  (\ref{A50})  this last equality occurs only if $j+1$ is even and $i=(j+1)/2$. Equation  (\ref{A50})  is then easily solved by setting
$$\ti a_{\e,i}=\int_{\R/\Z}h_{\e,i,i}(t)dt,\qquad g_{\e,i,i}(t)=-\int_{0}^t(h_{\e,i,i}(s)-\ti a_{\e,i})ds.$$
Equations (\ref{A49}) always admit unique 1-periodic solutions of the form
$$\begin{cases}&g_{\e,i_{1},i_{2}}(t)=e^{\l_{\e}(i_{1}-i_{2}) t}c_{\e,i_{1},i_{2}}-\int_{0}^t e^{(t-s)\l_{\e}(i_{1}-i_{2})}h_{\e,i_{1},i_{2}}(s)ds\\
&\textrm{where}\quad c_{\e,i_{1},i_{2}}=(e^{\l_{\e}(i_{1}-i_{2})}-1)^{-1}\int_{0}^1 e^{(1-s)\l_{\e}(i_{1}-i_{2})}h_{\e,i_{1},i_{2}}(s)ds.
\end{cases}$$
In the preceding solutions, the dependence on $\e$ is smooth and if, for $\e=0$,  the functions $h_{0,i_{1},i_{2}}$ do not depend on $t$ we see that $g_{0,t_{1},t_{2}}$ is a constant. 

This concludes the proof of Lemma \ref{lemma:Birkhoff}. \hfill $\Box$

\section{Extension of symplectic diffeomorphisms}
\begin{lemma}\label{lemma:extension}Let $(\Theta_{\e})_{\e\in ]-\e_{0},\e_{0}[}$, be a smooth (or continuous) family of $C^k$ symplectic diffeomorphisms  $C^1$-close to the identity, defined on some open disk $D(o,\d)$ of $\R^2$ and  such that $\Theta_{\e}(o)=o$. Then, there exists $(\ti \Theta_{\e})_{\e\in ]-\e_{0},\e_{0}[}$ a smooth (or continuous) family of $C^k$ symplectic diffeomorphisms of $\R^2$ such that on $D(o,\d/2)$ one has $\ti\Theta_{\e}=\Theta_{\e}$.
\end{lemma}
\begin{proof}We use the notation $\Theta_{\e}(x,y)=(\ti x,\ti y)$. Since $\Theta_{\e}$ is symplectic the 1-form  $\ti y d\ti x-ydx$ is closed and defined on a disk $D(o,4\d/5)$ of center $o$ and radius $4\d/5$ (we assume $\Theta_{\e}$ $C^1$-close to the identity so that we can use the Implicit Function Theorem). It is hence locally exact and there exists a function $S_{\e}(y,\ti y)$ such that $\ti y d\ti x-ydx=dS_{\e}$. Now  the function  $F_{\e}(x,\ti y)=-S_{\e}(y,\ti y)+(\ti x-x)\ti y$ is defined on $D(0,3\d/4)$ and satisifes
$( y-\ti y)dx+(\ti x- x)d\ti y=dF_{\e}$ or equivalently
\be \Theta_{\e}(x,y)=(\ti x, \ti y)\quad \iff\quad \begin{cases}&\ti x=x+\pa_{\ti y}F_{\e}(x,\ti y)\\ &y=\ti y+\pa_{x}F_{\e}(x,\ti y) .\end{cases}\label{eq:A.26}\ee
Note that we can choose choose $(F_{\e})_{\e}$ as a $C^k$-family of $C^{k+1}$-functions such that $F_{\e}(o)=0$, $DF_{\e}(o)=0$.

We can then choose $\chi:\R^2\to\R$ a smooth function which is equal to 1 on $D(o,2\d/3)$ and 0 outside $D(o,3\d/4)$, set
$$\ti F_{\e}=\chi\times F_{\e}$$
and define $\ti\Theta_{\e}$ by (\ref{eq:A.26}) with $F_{\e}$ replaced by $\ti F_{\e}$. The family of diffeomorphisms $(\ti \Theta_{\e})_{\e}$ is a smooth (or continuous) family of  exact symplectic $C^k$-diffeomorphisms.
\end{proof}

\section{Proof of Lemma \ref{lemma:9.1}}\label{sec:A3}
Let  $\chi:\R\to [0,1]$ be a smooth even  function with support in $[-1/2,1/2]$  such that $\chi(0)=1$ and which is increasing on $[-1/2,0]$. There exists $\a\in ]0,1/4[$ such that for all $x\in ]-2\a,2\a[ $ one has $\chi(x)>1/2$ and
$$\b_{\min}:=\min_{[-2\a,-\a]}\chi'>0,\qquad \b_{\max}:=\max_{[-2\a,-\a]}\chi'>0.
$$
We define for $\rho\in ]0,1/12]$ and $C_{M}>0$ 
$$\ph_{M}(x)=a(\rho,C_{M})\chi\biggl(\frac{x-1/3}{1/12}\biggr)-C_{M} \chi\biggl(\frac{x-2/3}{\rho}\biggr),
$$
where $a(\rho,C_{M})>0$ is chosen so that  
$$\int_{0}^1e^{\ph_{M}(u)}du=1.$$

Let $I_{}=(2/3)+]-2\a\rho,-\a\rho[$. For $x\in I_{}$ one has 
\begin{align*}&\ph_{M}(x)\leq -C_{M}/2=-C_{M}\a\b_{min}/(2\a\b_{min})\\
&\ph_{M}'(x)\leq -(C_{M}/\rho)\b_{min}=-(C_{M}\a\b_{min})/(\a\rho)=-C_{M}\a\b_{min}/|I_{}|\\
&\ph_{M}'(x)\geq -(C_{M}/\rho)\b_{max}=-(C_{M}\a\b_{max})/(\a\rho)=-(\b_{max}/\b_{min})C_{M}\a\b_{min}/|I_{}|.
\end{align*}
Fixing $\rho$ (for example $\rho=1/12$) and taking 
$$ b^{-1}=\max\biggl(\frac{\b_{max}}{\b_{min}},2\a\b_{min}\biggr),\qquad C_{M}=\frac{M}{\a\b_{min}},$$  provides the first two items of Lemma \ref{lemma:9.1}.

Let us check  the third item is satisfied. From the definition of $s_{M}$ one has $s_{M}'(x)=e^{\ph_{M}(x)}=1$ for $x\notin [0,1]$. Since $s_{M}(0)=0$ one has $s_{M}(x)=x$ for $x\leq 0$. Similarly since 
$$s_{M}(1)=\int_{0}^1e^{\ph_{M}(u)}du=1$$
we have $s_{M}(x)=x$ for $x\geq 1$.

Since in any case $s'(x)>0$, this concludes the proof of Lemma \ref{lemma:9.1}.\hfill $\Box$

\end{document}